\documentclass[12pt]{amsart}

\usepackage{amssymb}
\usepackage{amsmath}
\usepackage{enumerate}
\usepackage{amscd}

\newtheorem{introthm}{Theorem}

\newtheorem{theorem}{Theorem}[section]
\newtheorem{lemma}[theorem]{Lemma}
\newtheorem{corollary}[theorem]{Corollary}
\newtheorem{proposition}[theorem]{Proposition}
\newtheorem{conjecture}[theorem]{Conjecture}

\newcommand{\Z}{\mathbb{Z}}
\newcommand{\Q}{\mathbb{Q}}
\newcommand{\Field}{\mathbb{F}}

\newcommand{\betti}[2]{\beta_{#2}^{#1}}
\newcommand{\abetti}[2]{\alpha_{#2}^{#1}}
\newcommand{\tensor}{\otimes}
\newcommand{\join}{*}
\newcommand{\bigdirsum}{\bigoplus}
\newcommand{\sgn}{{\rm sgn}}
\newcommand{\brackom}[2]{\genfrac{[}{]}{0pt}{}{#1}{#2}}

\newcommand{\M}[1][{}]{\mathsf{M}_{#1}}
\newcommand{\Symm}[1]{\mathfrak{S}_{#1}}
\newcommand{\zivaljevic}{$\check{\mathrm{Z}}$ivaljevi\'c}
\newcommand{\shareshian}{Shareshian}

\begin{document}

 \title{Exact Sequences for the Homology of the Matching
   Complex}
 \thanks{This research was financed by a
   grant sponsored by Professor G\"unter M. Ziegler
   via his ``F\"orderpreis f\"ur deutsche Wissenschaftler im
   Gottfried Wilhelm Leibniz-Programm der Deutschen
   Forschungsgemeinschaft.''}
  \author{Jakob Jonsson}
   \email{jakobj@math.kth.se}
   \address{Department of Mathematics, KTH, 10044 Stockholm, Sweden}
  \date{}

  \begin{abstract}
    Building on work by Bouc and by {\shareshian} and Wachs, we
    provide a toolbox of long exact 
    sequences for the reduced simplicial homology of the matching
    complex $\M[n]$, which is the simplicial complex of matchings in
    the complete graph $K_n$. 
    Combining these sequences in
    different ways, we prove several results about the
    $3$-torsion part of the homology of $\M[n]$. 
    First, we demonstrate that there is nonvanishing
    $3$-torsion in $\tilde{H}_d(\M[n];\Z)$ whenever $\nu_n
    \le d \le \lfloor\frac{n-6}{2}\rfloor$, where $\nu_n= \lceil
    \frac{n-4}{3}\rceil$. 
    By results due to Bouc and to {\shareshian} and Wachs,
    $\tilde{H}_{\nu_n}(\M[n];\Z)$ is a nontrivial elementary
    $3$-group for almost all $n$ and the bottom nonvanishing homology 
    group of $\M[n]$ for all $n \neq 2$.
    Second, we prove that 
    $\tilde{H}_d(\M[n];\Z)$ is a nontrivial $3$-group whenever
    $\nu_n \le d \le \lfloor \frac{2n-9}{5}\rfloor$. 
    Third, for each $k \ge 0$, we show that there is a polynomial
    $f_k(r)$ of degree $3k$ such that the dimension of
    $\tilde{H}_{k-1+r}(\M[2k+1+3r];\Z_3)$, viewed as a vector space
    over $\Z_3$, is at most $f_k(r)$ for all $r \ge k+2$.
  \end{abstract}

  \maketitle

  \noindent
  This is a preprint version of a paper published in 
  {\em Journal of Combinatorial Theory, Series A}  {\bf 115} (2008),
  1504-1526.

  \section{Introduction}

  Given a family $\Delta$ of graphs on a fixed vertex set, we 
  identify each member of $\Delta$ with its edge set. In particular,
  if $\Delta$ is closed under deletion of edges, then 
  $\Delta$ is an abstract simplicial complex.

  A {\em matching} in a simple graph $G$ is a subset $\sigma$ of the
  edge set of $G$ such that no vertex appears in more than one edge in
  $\sigma$. Let $\M(G)$ be the family of matchings in $G$; $\M(G)$ is
  a simplicial complex. We write $\M[n] = \M(K_n)$, where $K_n$ is the
  complete graph on the vertex set 
  $[n] = \{1, \ldots, n\}$.

  The topology of $\M[n]$ and related complexes has been subject to
  analysis
  in a number of theses
  \cite{Andersen,Dong,Garst,thesis,thesislmn,Kara,Kson}
  and papers 
  \cite{Ath,BBLSW,BLVZ,Bouc,DongWachs,FH,KRW,RR,ShWa,Ziegvert}; see
  Wachs \cite{Wachs} for an excellent survey and further references.

  Despite the simplicity of the definition, the homology of the
  matching complex $\M[n]$
  turns out to have a complicated structure. The rational homology 
  is well-understood and easy to describe thanks to a beautiful result
  due to Bouc \cite{Bouc}, but very
  little is known about the integral homology and the homology over
  finite fields. 

  Over the integers, the bottom nonvanishing reduced homology group of
  $\M[n]$ is known to appear in degree $\nu_n =
  \lceil\frac{n-4}{3}\rceil$
  and is an elementary $3$-group for almost all $n$. 
  For $n \equiv 1 \pmod{3}$, this result is due to Bouc \cite{Bouc},
  who proved that $\tilde{H}_{r-1}(\M[3r+1];\Z) \cong \Z_3$
  for $r \ge 2$; see Section~\ref{newboucproof-sec}. {\shareshian} and
  Wachs \cite{ShWa} settled the 
  general case, proving that $\tilde{H}_{\nu_n}(\M[n];\Z) \cong
  (\Z_3)^{e_n}$ for some $e_n \ge 1$ whenever 
  $n \ge 15$ or $n \in \{7,10,12,13\}$;
  see Section~\ref{generalcase-sec}. Regarding the exact 
  value of $e_n$ when $n \not\equiv 1 \pmod{3}$, the best previously
  known upper bound is superexponential in $n$ \cite{ShWa}. In 
  Section~\ref{boundsmatch-sec}, we improve on this
  bound as follows:  
  \begin{introthm}
    We have that $e_{3r+3}$ 
    is bounded by a polynomial in $r$ of degree three
    and that $e_{3r+5}$ is bounded
    by a polynomial in $r$
    of degree six.
    More generally, for every $k \ge 0$, the dimension of the 
    $\Z_3$-vector space
    $\tilde{H}_{k-1+r}(\M[2k+1+3r];\Z_3)$ is bounded by a polynomial
    in $r$ of degree $3k$.
    \label{boundsintro-thm}
  \end{introthm}
  To establish Theorem~\ref{boundsintro-thm}, we construct a new long
  exact sequence for
  the matching complex, relating the homology of $\M[n] \setminus e$
  to that of $\M[n-2] \setminus e$, $\M[n-3]$, and $\M[n-5]$, where
  $e$ is an edge and $\M[n] \setminus e$ is the complex obtained from
  $\M[n]$ by removing the $0$-cell corresponding to this edge. See
  Section~\ref{exseq0235-sec} for details.
  Combining this
  sequence with the long exact sequence for the pair 
  $(\M[n], \M[n] \setminus e)$ (see Section~\ref{exseq0e2-sec}) and
  using an induction argument, we derive bounds of the form 
  \[
  \hat{\beta}_{k,r} \le \hat{\beta}_{k,r-1} + C_k r^{3k-1}(1+O(1/r)),
  \]
  where $\hat{\beta}_{k,r} = \dim_{\Z_3}
  \tilde{H}_{k-1+r}(\M[2k+1+3r];\Z_3)$. Summing over $r$, we obtain
  the desired result.

  As it turns out, for any fixed $k \ge 0$ and for sufficiently large
  $r$, we have that $\tilde{H}_{k-1+r}(\M[2k+1+3r];\Z)$ is a
  nontrivial $3$-group. In fact, we prove the following result in
  Section~\ref{pnot3-sec}:
  \begin{introthm}
    $\tilde{H}_d(\M[n];\Z)$ is a nontrivial $3$-group whenever
    $\lceil\frac{n-4}{3}\rceil 
    \le d \le 
    \lfloor\frac{2n-9}{5}\rfloor$.
    \label{p5intro-thm}
  \end{introthm}
  The groups being finite in the given interval is a consequence of
  Bouc's formula for the rational homology \cite{Bouc}; see
  Section~\ref{classic-sec}.
  To settle the nonexistence of $p$-torsion in $\tilde{H}_d(\M[n];\Z)$
  for $p \neq 3$, we use three long exact sequences. Bouc
  \cite{Bouc} introduced two of these sequences, one of which relates
  $\M[n]$ to $\M[n-1]$ and $\M[n-2]$ and the other 
  $\M[n]$ to $\M[n-3]$ and $\M[n-4]$; see Sections~\ref{exseq012-sec}
  and \ref{exseq034-sec}, respectively. The third sequence is new but
  based on the same idea and relates 
  $\M[n]$ to $\M[n-3]$, $\M[n-5]$, and $\M[n-6]$; see
  Section~\ref{exseq0356-sec}.

  These three sequences are all special cases of a more general
  construction involving a filtration of $\M[n]$ with respect to a
  given parameter $m \in [n]$: 
  \[
  \Delta^0_n \subseteq \Delta^1_n \subseteq \cdots \subseteq
  \Delta^{\min \{m,n-m\}}_n = \M[n].
  \]
  We obtain $\Delta^i_n$ from $\M[n]$ by removing all
  matchings containing at least $i+1$ edges $ab = \{a,b\}$ such that
  $a \in [m]$ and $b \in [m+1,n] = \{m+1, m+2, \ldots, n-1, n\}$. 
  It is a straightforward exercise to show that
  the relative homology of $(\Delta^i_n,\Delta^{i-1}_n)$ is isomorphic
  to a direct sum of homology groups of $\M[m-i] \join \M[n-m-i]$, 
  where $\join$ denotes simplicial join. For $m \in \{1,2\}$, the
  construction boils down to Bouc's two exact sequences, whereas the
  parameter choice $m=3$ yields our new exact sequence.
  For larger $m$, one would need more than one exact sequence to
  fully describe the correlations between the different matching
  complexes involved. 
  See
  Section~\ref{filtration-sec} for basic properties of the filtration.

  The group $\tilde{H}_d(\M[n];\Z)$ being nontrivial 
  when $d$ falls within the bounds of Theorem~\ref{p5intro-thm}
  is a consequence of the following result, which we prove in 
  Section~\ref{highermatch-sec}:
  \begin{introthm}
    For $n \ge 1$, there is nonvanishing $3$-torsion
    in $\tilde{H}_d(\M[n]; \Z)$
    whenever $\lceil\frac{n-4}{3}\rceil \le d \le
    \lfloor\frac{n-6}{2}\rfloor$. 
    In particular, 
    $\tilde{H}_d(\M[n]; \Z)$ is nonzero 
    if and only if
    $\lceil\frac{n-4}{3}\rceil \le d \le
    \lfloor\frac{n-3}{2}\rfloor$.
    \label{nonvan3intro-thm}
  \end{introthm}
  To prove the first statement in
  Theorem~\ref{nonvan3intro-thm}, we only need Bouc's original two
  sequences and the results of Bouc and of {\shareshian} and Wachs
  about the bottom nonvanishing homology. The second statement is a
  consequence of the first statement and Bouc's formula for the
  rational homology of $\M[n]$.

  In Section~\ref{newboucproof-sec}, we find another application of
  the new long exact sequence introduced in Section~\ref{exseq0356-sec}
  as we present a new proof of Bouc's result that
  $\tilde{H}_{r-1}(\M[3r+1];\Z) \cong \Z_3$ for $r \ge 2$.

  So far, all our results have been about the existence of $3$-torsion
  and the nonexistence of other torsion. Almost nothing is known about
  $p$-torsion when $p \neq 3$, but in a previous paper \cite{m14}, the
  author used a result due to Andersen \cite{Andersen}
  to prove that
    $\tilde{H}_{4}(\M[14]; \Z)$ is a finite nontrivial group of
    exponent a multiple of $15$.
  We have not been able to detect $5$-torsion in any other homology group
  $\tilde{H}_d(\M[n];\Z)$, but in Section~\ref{further5-sec}, we show
  that the case $5d = 2n-8$ is crucial for the general behavior:
  \begin{introthm}
    For $q \ge 3$,
    if $\tilde{H}_{2q}(\M[5q+4];\Z)$ contains nonvanishing 
    $5$-torsion, then so does
    $\tilde{H}_{2q+u}(\M[5q+4+2u];\Z)$ for each
    $u \ge 0$. In particular, if 
    $\tilde{H}_{2q}(\M[5q+4];\Z)$ contains nonvanishing 
    $5$-torsion for each $q \ge 3$,
    then so does $\tilde{H}_d(\M[n];\Z)$ whenever
    $\lceil \frac{2n-8}{5}\rceil \le 
    d \le \lfloor \frac{n-7}{2}\rfloor$.
    \label{further5intro-thm}
  \end{introthm}

  \begin{figure}[htb]
    \begin{center}
        \begin{tabular}{|r||c|c|c|c|c|c|}
          \hline
          \!$\tilde{H}_d(\M[n];\Z)$
          &  $d=0$ & \ 1 \ & \ 2 \ & \ 3 \ & \ 4 \ & \ 5 \
          \\
          \hline
          \hline
          $n = 3$ & $\Z^2$ & - & - & - & - & - \\
          \hline
          $4$ & $\Z^2$ & - & - & - & - & - \\
          \hline
          $5$ & - & $\Z^6$ & - & - & - & - \\
          \hline
          $6$ & - & $\Z^{16}$ & - & - & - & - \\
          \hline
          $7$ & - & $\Z_3$ & $\Z^{20}$ & - & - &  - \\
          \hline
          $8$  & - & - & $\Z^{132}$ & - & - & - \\
          \hline
          $9$  & - & - & $\Z_3^8 \oplus \Z^{42}$ &
          $\Z^{70}$  & - & - \\
          \hline
          $10$  & - & - & $\Z_3$ &
          $\Z^{1216}$  & - & - \\
          \hline
          $11$  & - & - & - & $\Z_3^{45} \oplus
          \Z^{1188}$ & $\Z^{252}$ & -  \\
          \hline
          $12$  & - & - & - & $\Z_3^{56}$  &
          $\Z^{12440}$& - \\
          \hline
          $13$  & - & - & - & $\Z_3$  &
          $T_1 \oplus \Z^{24596}$& $\Z^{924}$ \\
          \hline
          $14$  & - & - & - & -  &
          $T_2$ &  $\Z^{138048}$ \\
          \hline
        \end{tabular}
    \end{center}
    \caption{The homology of $\M[n]$ for $n \le 14$.
      $T_1$ and $T_2$ are
      nontrivial finite groups of exponent a multiple of $3$ and $15$,
      respectively; see Proposition~\ref{m13-prop} and
      Theorem~\ref{m14-thm}.} 
    \label{matching-fig}
  \end{figure}

  See Figure~\ref{matching-fig} for the homology of $\M[n]$ for $n \le
  14$. Many values were obtained via computer calculations
  \cite{BBLSW}; we have yet to find a computer-free method for
  calculating $\tilde{H}_d(\M[n];\Z)$ in the case that the group is 
  not free and not of size $3$. 

  \subsection{Notation}
  \label{basic-sec}
  
  For a finite set $S$, we let $\M[S]$ denote the matching complex on
  the complete graph with vertex set $S$. In particular, 
  $\M[{[n]}] = \M[n]$, where $[n] = \{1, \ldots, n\}$.
  For integers $a \le b$, we write $[a,b] = \{a, a+1, \ldots, b-1,
  b\}$.

  The {\em join} of two families of sets $\Delta$ and $\Sigma$,
  assumed to be defined on disjoint ground sets, is the family
  $\Delta \join \Sigma = \{ \delta \cup \sigma: \delta \in \Delta,
  \sigma \in \Sigma\}$.

  Whenever we discuss the homology of a simplicial complex or the
  relative homology of a pair of simplicial complexes, we mean reduced
  simplicial homology. 
  For a simplicial complex $\Sigma$ and a coefficient ring
  $\Field$, we denote the generator of $\tilde{C}_d(\Sigma; \Field)$
  corresponding to a set $\{e_0, \ldots, e_d\} \in \Sigma$ as
  $e_0 \wedge \cdots \wedge e_d$. 
  Given a cycle $z$ in a chain group $\tilde{C}_d(\Sigma;
  \Field)$, whenever we talk about $z$ as an element in the 
  induced homology group $\tilde{H}_d(\Sigma;\Field)$, we really mean
  the homology class of $z$.

  We will often consider pairs of
  complexes $(\Gamma, \Delta)$ such that
  $\Gamma \setminus \Delta$ is a union of families of the
  form
  \[
  \Sigma = \{\sigma\} \join \M[S], 
  \]
  where $\sigma = \{e_1, \ldots, e_s\}$ is a set of pairwise disjoint
  edges and $S$ is a subset of $[n]$ such that $S \cap e_i =
  \emptyset$ for each $i$.
  We may write the chain complex of $\Sigma$ as
  \[
    \tilde{C}_d(\Sigma;\Field) = (e_1 \wedge \cdots \wedge e_s)\Field
    \tensor_\Field \tilde{C}_{d-s}(\M[S];\Field),
  \]
  defining the boundary operator as 
  \[
  \partial(e_1 \wedge \cdots \wedge e_s \tensor_\Field c)
  = (-1)^s e_1 \wedge \cdots \wedge e_s \tensor_\Field \partial(c).
  \]
  For simplicity, we will often suppress $\Field$ from notation. For
  example, by some abuse of notation, we will write
  \[
  \langle e_1 \wedge \cdots \wedge e_s\rangle \tensor
  \tilde{C}_{d-s}(\M[S]) =
  (e_1 \wedge \cdots \wedge e_s)\Field \tensor_\Field
  \tilde{C}_{d-s}(\M[S];\Field).
  \]

  We say that a cycle $z$ in 
  $\tilde{C}_{d-1}(\M[n];\Field)$ has {\em type}
  $\brackom{n_1}{d_1} \wedge \cdots \wedge
  \brackom{n_s}{d_s}$ 
  if there is a partition  
  $[n] = \bigcup_{i=1}^s S_i$
  such that size of $S_i$ is $n_i$ and such that
  $z = z_1 \wedge \cdots \wedge z_s$,
  where $z_i$ is a cycle in $\tilde{C}_{d_i-1}(\M[S_i];\Field)$ for each
  $i$. We define a {\em refinement} of a type in the natural manner; 
  $\brackom{n_1}{d_1} \wedge \cdots \wedge
  \brackom{n_{s-2}}{d_{s-2}} \wedge \brackom{n_{s-1}}{d_{s-1}} \wedge
  \brackom{n_s}{d_s}$ is a refinement of  
  $\brackom{n_1}{d_1} \wedge \cdots \wedge \brackom{n_{s-2}}{d_{s-2}}
  \wedge \brackom{n_{s-1}+n_s}{d_{s-1}+d_s}$ and so on. We write
  $T \prec T'$ to denote that the type $T$ is a refinement of the type
  $T'$. If $z$ is of type $T$ and $T \prec T'$, then $z$ is also of
  type $T'$.
  Finally, we write
  $\brackom{n}{d}^2 = \brackom{n}{d} \wedge \brackom{n}{d}$, 
  $\brackom{n}{d}^3 = \brackom{n}{d}\wedge \brackom{n}{d}\wedge
  \brackom{n}{d}$, and so on.

  When dealing with the group $\tilde{H}_d(\M[n];\Z)$, we will find
  the following transformation very useful:
  \begin{equation}
    \left\{
    \begin{array}{l}
      k = 3d-n+4
      \\
      r = n-2d-3 
    \end{array}
    \right. \Longleftrightarrow
    \left\{
    \begin{array}{l}
    n = 2k+1+3r \\
    d = k-1+r.
    \end{array}
    \right.   
    \label{nd2kr-eq}
  \end{equation}
  In particular, we have the equivalences
  \[
  \left\lceil\frac{n-4}{3}\right\rceil \le d \le 
  \left\lfloor\frac{n-3}{2}\right\rfloor
  \Longleftrightarrow
  2d+3 \le n \le 3d+4 \Longleftrightarrow 
  \left\{
  \begin{array}{l}
    k \ge 0 \\
    r \ge 0.
  \end{array}
  \right.   
  \]
  For $n \ge 1$, Theorem~\ref{nonvan3intro-thm} yields that
  $\tilde{H}_d(\M[n];\Z)$ is nonzero if and only if these
  inequalities are satisfied.

  \subsection{Two classical results}
  \label{classic-sec}

  Before proceeding, we list two classical results pertaining to 
  the topology of the matching complex. 
  \begin{theorem}[Bouc \cite{Bouc}]
    For $n \ge 1$, the homology group $\tilde{H}_d(\M[n];\Q)
    = \tilde{H}_{k-1+r}(\M[2k+1+3r];\Q)$ is nonzero if and only if 
    \[
    \left\lceil\frac{n-\lfloor\sqrt{n}\rfloor-2}{2}\right\rceil \le d
    \le \left\lfloor\frac{n-3}{2}\right\rfloor
    \Longleftrightarrow 
    \left\{
    \begin{array}{l}
      k \ge \binom{r}{2} \\
      r \ge 0.
    \end{array}
    \right.   
    \]
  \label{bouc-thm}
  \end{theorem}
  Theorem~\ref{bouc-thm} is an immediate consequence of a concrete
  formula for the rational homology of $\M[n]$; see Bouc \cite{Bouc}
  for details and Wachs' survey \cite{Wachs} for an overview.
  \begin{theorem}[Bj\"{o}rner
      {et al.} \cite{BLVZ}]
    For $n \ge 1$, $\M[n]$ is $(\nu_n-1)$-connected, 
    where $\nu_n = \lceil\frac{n-4}{3}\rceil$. 
    \label{smallest-thm}
  \end{theorem}
  Indeed, the $\nu_n$-skeleton of $\M[n]$ is shellable \cite{ShWa} and
  even vertex decomposable \cite{Ath}. As already mentioned in the
  introduction, there is nonvanishing homology in 
  degree $\nu_n$ for all $n \neq 2$; see Section~\ref{bottommatch-sec}
  for details.

  \section{Filtration of $\M[n]$ with respect to a fixed vertex
          set}
  \label{filtration-sec}

  The following general construction forms the basis of the three
  exact sequences presented in Sections~\ref{exseq012-sec},
  \ref{exseq034-sec}, and \ref{exseq0356-sec}. The first two sequences
  already appeared in the work of Bouc \cite{Bouc}, whereas the third
  one is new.

  Given a vertex set $S \subseteq [n]$, form a sequence
  \[
  \Delta^0_n \subseteq \Delta^1_n \subseteq \cdots \subseteq
  \Delta^{\min \{\# S,n-\# S \}}_n
  \]
  of simplicial complexes, where 
  we obtain $\Delta^{i-1}_n$ from $\M[n]$ by removing all matchings
  $\sigma$ containing at least $i$ edges $ab$ such that $a \in S$ and
  $b \in [n] \setminus S$. We also define $\Delta^{-1}_n =
  \emptyset$.
  Assuming that $S = [m]$, one easily checks that
  \begin{equation}
  \Delta^i_n \setminus \Delta^{i-1}_n
  = \bigcup 
  \{\{a_1b_1, \ldots, a_ib_i\}\} \join
  \M[{[m]\setminus A}]
  \join
  \M[{[m+1,n]\setminus B}],
  \label{matchunion-eq}
  \end{equation}
  where the union is over all pairs of sequences $(a_1,
  \ldots, a_i)$ and $(b_1, \ldots, b_i)$ of distinct elements such
  that $1 \le a_1 < \cdots < a_i \le m$ and
  $b_1, \ldots, b_i \in [m+1,n]$; $A = \{a_1, \ldots, a_i\}$ and $B =
  \{b_1, \ldots, b_i\}$. The families in the union form
  an antichain under inclusion, meaning that if $\sigma$ belongs to
  one of the families and $\tau$ to another, then $\sigma
  \not\subseteq \tau$ and $\tau \not\subseteq \sigma$.
  One readily verifies that this implies the following:
  \begin{lemma}
    For $0 \le i \le \min \{m,n-m\}$ and all $d$, we have that
    \[
    \tilde{H}_d(\Delta^i_n,\Delta^{i-1}_n)
    \cong \bigoplus
    \langle a_1b_1 \wedge \cdots \wedge a_ib_i\rangle \tensor
    \tilde{H}_{d-i}(\M[{[m]\setminus A}]
    \join
    \M[{[m+1,n]\setminus B}]),
    \]
    where 
    the direct sum is over all pairs of ordered sequences $(a_1,
    \ldots, a_i)$ and $(b_1, \ldots, b_i)$ with properties as above. 
    \label{relative-lem}
  \end{lemma}
  As a consequence, we have a long
  exact sequence of the form
  \[
  \begin{CD}
    & & & & \cdots @>>>\! 
    \displaystyle{\bigdirsum_t}
    \tilde{H}_{d-i+1}(\M[m-i] \join \M[n-m-i]) \\
    @>>>\! 
    \tilde{H}_{d}(\Delta_n^{i-1}) @>>>\!
    \tilde{H}_{d}(\Delta_n^i) @>>>\!
    \displaystyle{\bigdirsum_t}
    \tilde{H}_{d-i}(\M[m-i] \join \M[n-m-i]) \\
    @>>>\!
    \tilde{H}_{d-1}(\Delta_n^{i-1}) @>>>\! \cdots ,
  \end{CD}
  \]
  where $t = i!\binom{m}{i}\binom{n-m}{i}$.
  For $m-i \le 3$,
  the situation is particularly simple, as $\M[m-i]$ is then either
  the empty complex $\{\emptyset\}$, a single point, or three isolated
  points. We will exploit this fact in Sections~\ref{exseq012-sec},  
  \ref{exseq034-sec}, and \ref{exseq0356-sec}.

  \section{Five long exact sequences}

  We present five long exact sequences relating different
  families of matching complexes. Throughout this
  section, we consider an arbitrary coefficient ring $\Field$ with unit,
  which we suppress from notation for convenience.

  \subsection{Long exact sequence relating $\M[n]$, $\M[n-1]$, and
          $\M[n-2]$}
  \label{exseq012-sec}

  The choice $m=1$ yields the simplest special case of the
  construction in Section~\ref{filtration-sec}.
  Inserting $i=0$ and $i=1$ in (\ref{matchunion-eq}), we obtain 
  families involving complexes isomorphic to $\M[n-1]$ and $\M[n-2]$,
  respectively. More exactly, we have the following result:
  \begin{theorem}[Bouc \cite{Bouc}]
    For each $n \ge 2$, we have a long exact sequence
    \[
    \begin{CD}
      & & & & \cdots \!@>>>\! 
      \displaystyle{\bigoplus_{s=2}^n} 
      \langle 1s \rangle \tensor \tilde{H}_{d}(\M[{[2,n] \setminus
      \{s\}}]) \\ 
      \!@>>>\! 
      \tilde{H}_{d}(\M[{[2,n]}])\! @>>>\!
      \tilde{H}_{d}(\M[n]) \!@>\omega>>\!
      \displaystyle{\bigoplus_{s=2}^n} 
      \langle 1s\rangle \tensor \tilde{H}_{d-1}(\M[{[2,n] \setminus
      \{s\}}])  \\ 
      \!@>>>\!
      \tilde{H}_{d-1}(\M[{[2,n]}]) \!@>>>\! \cdots ,
    \end{CD}
    \]
    where $\omega$ is induced by the natural
    projection map. 
    \label{exseq012-thm}
  \end{theorem}
  We refer to this sequence as the {\em 0-1-2 sequence}, thereby
  indicating that the sequence relates $\M[n-0]$, $\M[n-1]$,
  and $\M[n-2]$. 

  \subsection{Long exact sequence relating $\M[n]$, $\M[n-3]$, and
  $\M[n-4]$}
  \label{exseq034-sec}

  We proceed with the case $m=2$ of the construction in 
  Section~\ref{filtration-sec}.
  In this case, $i \in \{1,2\}$ inserted into (\ref{matchunion-eq})
  yields families involving complexes isomorphic to $\M[n-3]$ and
  $\M[n-4]$, whereas $i=0$ yields a family involving contractible
  complexes; $\M[2]$ is a point. This turns out to imply the following
  result: 
  \begin{theorem}[Bouc \cite{Bouc}]
    Let $n \ge 4$ and define
    \begin{eqnarray*}
      Q_d^{n-4} &=& \bigdirsum_{s \neq t \in [3,n]}
      \langle 1s \wedge 2t \rangle \tensor \tilde{H}_{d}(\M[{[3,n]\setminus
      \{s,t\}}]); \\ 
      R_d^{n-3} &=& \bigdirsum_{a=1}^2 \bigdirsum_{u=3}^n
      \langle au \rangle \tensor \tilde{H}_{d}(\M[{[3,n]\setminus \{u\}}]).
    \end{eqnarray*}
    Then we have a long exact sequence
    \[
    \begin{CD}
      & & & & \cdots @>>> 
      Q_{d-1}^{n-4} \\
      @>\psi^*>>  
      R_{d-1}^{n-3}
      @>\varphi^*>>
      \tilde{H}_{d}(\M[n])  
      @>\kappa^*>> 
      Q_{d-2}^{n-4}
      \\
      @>\psi^*>>  
      R_{d-2}^{n-3}
      @>>> \cdots ,
    \end{CD}
    \]
    where $\psi^*$ is induced by the map 
    $\psi : 1s \wedge 2t \tensor x \mapsto
    2t \tensor x - 1s \tensor x$,
    $\varphi^*$ is induced by the map 
    $\varphi : au \tensor x \mapsto (au-12)\wedge x$,
    and $\kappa^*$ is induced by the natural projection map.
    \label{exseq034-thm}
  \end{theorem}
  We refer to this sequence as the {\em 0-3-4 sequence}.

  \subsection{Long exact sequence relating $\M[n]$, $\M[n-3]$, $\M[n-5]$,
  and $\M[n-6]$}
  \label{exseq0356-sec}

  For our third application of the construction in 
  Section~\ref{filtration-sec}, we consider $m=3$.
  In this case, the relevant matching complexes are isomorphic to
  $\M[n-3], \M[n-5],$ and $\M[n-6]$. 

  As in Section~\ref{filtration-sec}, 
  we define $\Delta^i_n$ to be the complex of matchings $\sigma$ 
  such that at most $i$ of the vertices in $\{1,2,3\}$ 
  are matched in $\sigma \setminus \{12,13,23\}$. 

  \begin{lemma}
    Let $n \ge 5$. We have an isomorphism
    \[
    \varphi^* : P_{d-2}^{n-5} \oplus Q_{d-1}^{n-3}
    \rightarrow \tilde{H}_d(\Delta^2_n),
    \]
    where
    \[
    P_{d}^{n-5} = 
    \bigoplus_{1 \le a < b \le 3} 
    \bigoplus_{s \neq t \in [4,n]}
    \langle as\wedge bt \rangle \tensor \tilde{H}_{d}(\M[{[4,n] \setminus
    \{s,t\}}])
    \]
    and
    \[
    Q_{d}^{n-3} =
    \bigoplus_{c=2}^3 \langle 1c \rangle \tensor \tilde{H}_{d}(\M[{[4,n]}]).
    \]
    The isomorphism $\varphi^*$ is induced by the map $\varphi$
    defined by 
    $\varphi(1c \tensor x) = \varphi(1c) \wedge
    x$, where $\varphi(1c) = 1c - 23$, and
    $\varphi(as \wedge bt \tensor x) = 
    \varphi(as \wedge bt) \wedge x$, where 
    \[
    \varphi(as\wedge bt) = as\wedge bt
    + ac \wedge (st - bt) + bc \wedge (as - st)
    \]
    and $\{a,b,c\} = \{1,2,3\}$.
    \label{exseq0356-lem}
  \end{lemma}
  \begin{proof}
    First, we show that the sequence
    \begin{equation}
    \begin{CD}
      0  @>>>  
      \tilde{H}_d(\Delta^{1}_n) 
      @>>>  
      \tilde{H}_d(\Delta^{2}_n)
       @>>>  
      \tilde{H}_d(\Delta^{2}_n,\Delta^{1}_n) &
      @>>>  0
    \end{CD}
    \label{delta2seq-eq}
    \end{equation}
    is split exact for each $d$. 
    To see this, first note that 
      $\tilde{H}_d(\Delta^{2}_n,\Delta^{1}_n)
    \cong P_{d-2}^{n-5}$; apply Lemma~\ref{relative-lem}.
    Next, define $\hat{\varphi}$ to be the restriction
    of $\varphi$ to 
    $\tilde{C}_d(\Delta_n^2,\Delta_n^1)$
    and note that the projection of 
    $\hat{\varphi}(as \wedge bt \tensor x)$ on
    $\tilde{C}_i(\Delta^2_n,\Delta^1_n)$ is again 
    $as \wedge bt \tensor x$. Since $\hat{\varphi}$ clearly commutes
    with the boundary operator, the
    sequence (\ref{delta2seq-eq}) is split exact as desired.
    We conclude that we have an isomorphism
    \[
    \tilde{H}_d(\Delta^{2}_n) \cong 
    \tilde{H}_d(\Delta^{1}_n) \oplus
    P_{d-2}^{n-5}.
    \]

    It remains to prove that the restriction 
    of $\varphi^*$ to $Q_{d-1}^{n-3}$ defines an isomorphism 
    from $Q_{d-1}^{n-3}$ to $\tilde{H}_d(\Delta^{1}_n)$.
    By Lemma~\ref{relative-lem}, that would be true if we replaced
    $\Delta_n^1$ with $\Delta_n^0$. Thus to conclude the
    proof, it suffices to prove that the relative homology of the pair 
    $(\Delta_n^1,\Delta_n^0)$ vanishes. By Lemma~\ref{relative-lem},
    we obtain that
    \[
    \tilde{H}_d(\Delta_n^1,\Delta_n^0)
    \cong \bigoplus_{a=1}^3 \bigoplus_{u=4}^n 
    \langle au \rangle \tensor  \tilde{H}_{d-1}(\M[\{1,2,3\} \setminus 
    \{a\}] \join \M[{[4,n] \setminus \{u\}}]) = 0;
    \]
    the latter equality is a consequence of the fact that
    $\M[\{1,2,3\} \setminus \{a\}] \cong \M[2]$ is a point. 
  \end{proof}
  
  \begin{theorem}
    Let $n \ge 6$.
    Define $P_{d}^{n-5}$, $Q_d^{n-3}$, and $\varphi^*$ as in
    Lemma~\ref{exseq0356-lem} and  
    let 
    \[
    R_{d}^{n-6} = \bigoplus_{(s,t,u)} \langle 1s \wedge 2t \wedge 3u
    \rangle \tensor   
    \tilde{H}_{d}(\M[{[4,n] \setminus
    \{s,t,u\}}]),
    \]
    where the sum is over all triples of distinct integers $(s,t,u)$
    such that $s,t,u \in [4,n]$.
    Then we have a long exact sequence
    \[
    \begin{CD}
      & & & & \cdots
      @>>>  
      R_{d-2}^{n-6}  \\ 
      @>\psi^*>>  
      P_{d-2}^{n-5} \oplus Q_{d-1}^{n-3} 
      @>\iota^*\circ \varphi^*>>  
      \tilde{H}_d(\M[n])
      @>>>  
      R_{d-3}^{n-6}  \\
      @>\psi^*>>  
      P_{d-3}^{n-5} \oplus Q_{d-2}^{n-3}  
      @>>> 
      \cdots ,
    \end{CD}
    \]
    where  $\psi^*$ is induced by the map
    \begin{eqnarray*}
      \psi(1s \wedge 2t \wedge 3u \tensor x) 
      &=& 1s\wedge 2t \tensor x + 2t\wedge 3u \tensor x 
      - 1s \wedge 3u \tensor x \\ 
      &+& 
      12 \tensor (su-tu) \wedge x + 
      13 \tensor (tu-st) \wedge x
    \end{eqnarray*}
    and $\iota^*$ is induced by the natural inclusion map
    $\iota : \tilde{C}_d(\Delta^2_n) \rightarrow
    \tilde{C}_d(\M[n])$.
    \label{exseq0356-thm}
  \end{theorem}
  \begin{proof}
    By Lemma~\ref{relative-lem},
    $\tilde{H}_d(\Delta^{3}_n,\Delta^{2}_n) \cong R_{d-3}^{n-6}$.
    Hence by Lemma~\ref{exseq0356-lem}, it remains to prove that $\psi^*$
    has properties as stated in the theorem. For this, note that
    the natural map 
    \[
    \hat{\psi} : \tilde{C}_d(\Delta^{3}_n,\Delta^{2}_n)
    \rightarrow  \tilde{C}_d(\Delta_n^2)
    \]
    is given by 
    \[
    \hat{\psi}(1s \wedge 2t \wedge 3u)
    = \partial(1s \wedge 2t \wedge 3u) =
    1s\wedge 2t + 2t\wedge 3u - 1s\wedge 3u,
    \]
    suppressing ``$\tensor x$'' from notation.
    Moreover, note that 
    \begin{eqnarray*}
      & &\varphi(12 \tensor (su-tu) + 
      13 \tensor (tu-st)) \\
    &=& (12-23) \wedge (su-tu)
    + (13-23) \wedge (tu-st) \\
    &=&
    {}12 \wedge (su-tu) + 
    {}13 \wedge (tu-st) + 
    {}23 \wedge (st-su)
    \end{eqnarray*}
    and
    \begin{eqnarray*}
    & &\varphi(1s \wedge 2t +
    2t \wedge 3u -
    1s \wedge 3u) - \partial(1s\wedge 2t\wedge 3u)
    \\
    &=&
      13 \wedge (st-2t) +       23\wedge (1s-st) + 
      12 \wedge (tu-3u) + 13\wedge (2t-tu)\\
      &-&
      12\wedge (su-3u) - 23 \wedge (1s-su) \\
      &=& 
      12 \wedge (tu-su) + 
      13 \wedge (st-tu) + 
      23\wedge (su-st).
    \end{eqnarray*}
    Since $\psi$ is given by $\varphi^{-1} \circ \hat{\psi}$, 
    we are done.
  \end{proof}

  We refer to this sequence as the  {\em 0-3-5-6 sequence}.

  \begin{corollary}
    For each $n \ge 6$, we have the exact sequence
    \[
    \begin{CD}
      R_{\nu_n-2}^{n-6}  
      @>\psi^*>>  
      P_{\nu_n-2}^{n-5} \oplus
      Q_{\nu_n-1}^{n-3} 
      @>\iota^*\circ \varphi^*>>  
      \tilde{H}_{\nu_n}(\M[n])
       @>>> 0,
    \end{CD}
    \]
    where $\nu_n = \lceil\frac{n-4}{3}\rceil$.
    If $n \equiv 1 \pmod{3}$, then
    $P_{\nu_n-2}^{n-5} = 0$.
    \label{exseq0356-cor}
  \end{corollary}
  \begin{proof}
    This is immediate by
    Theorems~\ref{smallest-thm} and \ref{exseq0356-thm}.
  \end{proof}

  \subsection{Long exact sequence relating $\M[n]$, $\M[n]\setminus e$
    and $\M[n-2]$} 
  \label{exseq0e2-sec}

  We proceed with the long exact sequence for
  the pair $(\M[n],\M[n] \setminus e)$, where $e$ is any edge
  and $\M[n] \setminus e$ is the complex obtained by removing the
  $0$-cell $e$.
  \begin{theorem}
    For each $n \ge 2$ and each edge $e$ in the complete graph $K_n$,
    we have a long exact sequence
    \[
    \begin{CD}
      & & & & \cdots \!@>>>\! 
      \langle e\rangle \tensor \tilde{H}_{d}(\M[{[n]\setminus e}]) \\
      \!@>>>\! 
      \tilde{H}_{d}(\M[n]\setminus e)\! @>>>\!
      \tilde{H}_{d}(\M[n]) \!@>\omega>>\!
      \langle e \rangle \tensor \tilde{H}_{d-1}(\M[{[n]\setminus e}])  \\
      \!@>>>\!
      \tilde{H}_{d-1}(\M[n] \setminus e) \!@>>>\! \cdots ,
    \end{CD}
    \]
    where $\omega$ is induced by the natural
    projection map. 
    \label{exseq0e2-thm}
  \end{theorem}
  \begin{proof}
    Simply note that 
    $\M[n] \setminus (\M[n] \setminus e) = \{\{e\}\} \join
    \M[{[n]\setminus e}]$.
  \end{proof}
  We refer to this sequence as the {\em 0-$e$-2 sequence}.
  We will make use of this sequence when providing bounds on the
  homology in Section~\ref{boundsmatch-sec}.

  \subsection{Long exact sequence relating $\M[n] \setminus e$,
    $\M[n-2] \setminus e$, $\M[n-3]$, and $\M[n-5]$}
  \label{exseq0235-sec}

  Using an approach similar to the one in Section~\ref{exseq0356-sec},
  we construct a long exact sequence relating $\M[n] \setminus
  e$, $\M[n-2] \setminus e$, $\M[n-3]$, and $\M[n-5]$, where $e$ is
  any edge. The main benefit of this sequence is that it provides good
  bounds on the homology when combined with the sequence in
  Section~\ref{exseq0e2-sec}; see Section~\ref{boundsmatch-sec}.
  Since we will not make use of the homomorphisms in this exact
  sequence, we do not define them explicitly; the interested reader
  will note that they are straightforward, though a bit cumbersome, to
  derive from the proof.
  \begin{theorem}
    Let $n \ge 5$.
    Define 
    \[
    P_{d}^{n-5} = 
    \bigoplus_{s \neq t \in [4,n]}
    \langle 1s \wedge 2t \rangle \tensor \tilde{H}_{d}(\M[{[4,n]
    \setminus 
    \{s,t\}}])
    \]
    and
    \[
    Q_{d}^{n-2} = \bigoplus_{i=4}^n 
    \langle 3u \rangle \tensor \tilde{H}_{d}(\M[{[n] \setminus
    \{3,u\}}] \setminus 12).
    \]
    Then we have a long exact sequence
    \[
    \begin{CD}
      & & & & \cdots
      @>>>  
      Q_{d}^{n-2}  \\ 
      @>>>  
      \langle 13 \rangle \tensor \tilde{H}_{d-1}(\M[{[4,n]}]) \oplus
      P_{d-2}^{n-5}  
      @>>>  
      \tilde{H}_d(\M[n]\setminus 12)
       @>>>  
      Q_{d-1}^{n-2}  \\
      @>>>  
      \langle 13 \rangle \tensor \tilde{H}_{d-2}(\M[{[4,n]}]) \oplus
      P_{d-3}^{n-5}  
      @>>> 
      \cdots .
    \end{CD}
    \]
    \label{exseq0235-thm}
  \end{theorem}
  \begin{proof}
    Consider the long exact sequence for the pair $(\M[n] \setminus
    12,\Delta^2_n)$, where $\Delta^2_n$ is the complex obtained from
    $\M[n] \setminus 12$ by removing the elements $34, \ldots, 3n$. 
    Analogously to Lemma~\ref{relative-lem}, we have that
    $\tilde{H}_d(\M[n] \setminus 12, \Delta^2_n) \cong
    Q_{d-1}^{n-2}$.

    To settle the theorem, it suffices to prove that
    \[
    \tilde{H}_d(\Delta^2_n) \cong
    \langle 13 \rangle \tensor \tilde{H}_{d-1}(\M[{[4,n]}]) \oplus
    P_{d-2}^{n-5}. 
    \]
    To achieve this goal, define
    $\Delta_n^1$ to be the subcomplex of $\Delta^2_n$ obtained by
    removing all faces containing $\{1s,2t\}$ for some
    $s,t \in [4,n]$. 
    Analogously to Lemma~\ref{relative-lem}, we have that
    $\tilde{H}_d(\Delta^2_n,\Delta^1_n) \cong P_{d-2}^{n-5}$.
    A homomorphism $\varphi^*$ from $ P_{d-2}^{n-5}$ to $\Delta_n^2$
    is given by mapping $1s \wedge 2t$ to the cycle
    \[
    1s \wedge 2t + 2t \wedge 13 + 13 \wedge st + st \wedge 23 + 
    23 \wedge 1s.
    \]
    It is clear that the natural map back to $P_{d-2}^{n-5}$ has the
    property that $\varphi^*(1s \wedge 2t \tensor z)$ is mapped to 
    $1s \wedge 2t \tensor z$; hence we have a split exact sequence
    just as in 
    (\ref{delta2seq-eq}) in the proof of Lemma~\ref{exseq0356-lem}.
    This implies that
    \[
    \tilde{H}_d(\Delta^{2}_n) \cong 
    \tilde{H}_d(\Delta^{1}_n) \oplus
    P_{d-2}^{n-5},
    \]
    again as in the proof of Lemma~\ref{exseq0356-lem}, except
    that the complexes and groups are different.
    
    It remains to prove that $\tilde{H}_{d}(\Delta^1_n) \cong
    \langle 13 \rangle \tensor \tilde{H}_{d-1}(\M[{[4,n]}])$. 
    Let 
    $\Delta^0_n$ be the subcomplex of $\Delta^1_n$ obtained by 
    removing the elements $14, \ldots, 1n$ 
    and $24, \ldots, 2n$. Since
    \[
    \Delta^0_n = \{\emptyset, 13, 23\} \join \M[{[4,n]}],
    \]
    we obtain that
    $\tilde{H}_d(\Delta^0_n) \cong
    \langle 13\rangle  \tensor \tilde{H}_{d-1}(\M[{[4,n]}])$.
    Thus the only thing remaining is to prove that
    $\tilde{H}_d(\Delta^1_n)  \cong \tilde{H}_d(\Delta^0_n)$.
    Now,
    \[
    \Delta^1_n \setminus \Delta^0_n
    = \bigcup_{a=1}^2\bigcup_{u=4}^n 
    \{\{au\}\} \join \M[{\{3-a,3\}}] \join \M[{[4,n]
    \setminus \{u\}}], 
    \]
    which yields that
    \[
    \tilde{H}_d(\Delta^1_n, \Delta^0_n)
    \cong 
    \bigoplus_{a=1}^2\bigoplus_{u=4}^n  
    \langle au \rangle \tensor \tilde{H}_{d-1}(\M[{\{3-a,3\}}] \join
    \M[{[4,n] \setminus \{u\}}]) = 0;
    \]
    the homology of a cone vanishes.
  \end{proof}

  We refer to this sequence as the {\em 0-2-3-5 sequence}.

  \section{Bottom nonvanishing homology}
  \label{bottommatch-sec}

  We consider the bottom nonvanishing homology group 
  $\tilde{H}_{\nu_n}(\M[n];\Z)$, starting with the case $n \equiv 1
  \pmod{3}$  
  in Section~\ref{newboucproof-sec} and proceeding with the general
  case in Section~\ref{generalcase-sec}. 

  Before examining the different cases, we present a nice result due
  to {\shareshian} and Wachs about the structure of the bottom
  nonvanishing homology group of $\M[n]$. Using 
  the 0-3-5-6 sequence from Section~\ref{exseq0356-sec}, 
  we may provide a more streamlined proof for the case $n \equiv 2
  \pmod{3}$.

  Recall the concept of type introduced in Section~\ref{basic-sec}.
  \begin{lemma}[{\shareshian} \& Wachs \cite{ShWa}]
    For $k \in \{0,1,2\}$ and $r \ge 0$, the group
    $\tilde{H}_{k-1+r}(\M[2k+1+3r])$ is generated by
    cycles of type $\brackom{3}{1}^r \wedge \brackom{2k+1}{k}$.
    \label{generators-lem}
  \end{lemma}
  \begin{proof}
    For $k \in \{0,1\}$,
    {\shareshian} and Wachs \cite[Lemmas 2.3 \&
    2.5]{ShWa} provided a straightforward proof based on the tail end of
    the 0-3-4 sequence in Section~\ref{exseq034-sec}.
    Assume that $k = 2$ and write $n(k,r) = 2k+1+3r$ and $d(k,r) =
    k-1+r$; recall (\ref{nd2kr-eq}). The case $r=0$ is trivially true;
    hence assume that $r \ge 1$. The tail end in
    Corollary~\ref{exseq0356-cor} becomes
    \[
    \begin{CD}
      P_{d(1,r-1)}^{n(1,r-1)} \oplus
      Q_{d(2,r-1)}^{n(2,r-1)} 
      @>\iota^*\circ \varphi^*>>  
      \tilde{H}_{d(2,r)}(\M[n(2,r)])
       @>>> 0.
    \end{CD}
    \]
    By properties of $\iota^*\circ \varphi^*$,
    it follows that $\tilde{H}_{d(2,r)}(\M[n(2,r)])$ is generated by
    cycles of type
    $\brackom{5}{2}\wedge \brackom{n(1,r-1)}{d(1,r-1)+1}$
    and 
    $\brackom{3}{1}\wedge \brackom{n(2,r-1)}{d(2,r-1)+1}$.
    Now, a cycle of type $\brackom{n(1,r-1)}{d(1,r-1)+1}$ is a sum of
    cycles of type $\brackom{3}{1}^r$, whereas induction on $r$ yields
    that a cycle of type $\brackom{n(2,r-1)}{d(2,r-1)+1}$ is a sum of
    cycles of type $\brackom{3}{1}^{r-1} \wedge \brackom{5}{2}$.
  \end{proof}
  Lemma~\ref{generators-lem} does not generalize to arbitrary $k$. For
  example, for $(k,r) = (6,4)$, we obtain $\tilde{H}_{9}(\M[25];\Z)$,
  which is infinite by Theorem~\ref{bouc-thm}. In particular, this
  group cannot be generated by cycles of type  
  $\brackom{3}{1}^4 \wedge \brackom{13}{6} \prec 
  \brackom{12}{4} \wedge \brackom{13}{6}$, as these cycles all have
  finite exponent dividing three; $\tilde{H}_{3}(\M[12];\Z)$ is finite
  of exponent three.

  \subsection{The case $n \equiv 1 \pmod{3}$}
  \label{newboucproof-sec}

  For $r \ge 0$, define
  \begin{eqnarray}
    \gamma_{3r} &=& (12-23) \wedge (45-56) \wedge (78-89) 
    \nonumber
    \\
    & & \wedge \cdots \wedge ((3r-2)(3r-1)-(3r-1)(3r));
    \label{gamman-eq}
  \end{eqnarray}
  this is a cycle in both $\tilde{C}_{r-1}(\M[3r]; \Z)$ 
  and $\tilde{C}_{r-1}(\M[3r+1]; \Z)$.
  By Lemma~\ref{generators-lem}, 
  $\tilde{H}_{r-1}(\M[3r+1]; \Z)$ is generated by 
  $\{\pi(\gamma_{3r}) : \pi \in \Symm{3r+1}\}$, where 
  the action of $\Symm{3r+1}$ on $\tilde{H}_{r-1}(\M[3r+1]; \Z)$
  is the one induced by the natural action on the underlying vertex
  set $[3r+1]$.

  Using the long exact 0-3-5-6 sequence in
  Section~\ref{exseq0356-sec}, we give a new proof of a celebrated
  result due to Bouc about the bottom
  nonvanishing homology of $\M[n]$ for $n \equiv 1 \pmod{3}$.
  \begin{theorem}[Bouc \cite{Bouc}]
    For $r \ge 2$, we have that
    $\tilde{H}_{r-1}(\M[3r+1];\Z) \cong \Z_3$.
    \label{bouctor-thm}
  \end{theorem}
  \begin{proof}
    By Corollary~\ref{exseq0356-cor}, we have the exact
    sequence
    \[
    \begin{CD}
      R_{r-3}^{3r-5}  
      @>\psi^*>>  
      Q_{r-2}^{3r-2} 
      @>\iota^*\circ \varphi^*>>  
      \tilde{H}_{r-1}(\M[3r+1])
       @>>> 0.
    \end{CD}
    \]
    For $r=2$, this becomes
    \[
    \begin{CD}
      \displaystyle{\bigoplus_{s,t,u}}\  \langle 1s \wedge 2t\wedge 3u
      \rangle 
      @>\psi^*>>  
      \displaystyle{\bigoplus_{c=2}^3}\ 
      \langle 1c \rangle \tensor \tilde{H}_{0}(\M[{[4,7]}])
      @>\iota^*\circ \varphi^*>>  
      \tilde{H}_{1}(\M[7])
      \longrightarrow  0,
    \end{CD}
    \]
    where the first direct sum ranges over all triples of distinct
    vertices $s,t,u \in [4,7]$.
    A basis for $\M[{[4,7]}]$ is given by $\{45-56,46-56\}$; hence
    a basis for $Q_{0}^{4}$ is given by $\{e_{25}, e_{26}, e_{35},
    e_{36}\}$, where
    $e_{cd} = 1c \tensor (4d-56)$.
    Now, 
    \[
    \psi^*(1s \wedge 2t \wedge 3u)
    = 12 \tensor (su-tu) + 
    13 \tensor (tu-st);
    \]
    apply Theorem~\ref{exseq0356-thm} and
    Corollary~\ref{exseq0356-cor}. 
    In particular, if $\{s,t,u\} = \{4,5,6\}$, then
    \begin{eqnarray*}
    \psi^*(1s \wedge 2t \wedge 37)
    &=& 12 \tensor (s7-t7) + 
    13 \tensor (t7-st) \\
    &=& 12 \tensor (tu-su) + 
    13 \tensor (su-st) \\
    &=& \psi^*(1t \wedge 2s \wedge 3u).
    \end{eqnarray*}
    Similarly, $\psi^*(1s \wedge 27 \wedge 3u) = 
    \psi^*(1u \wedge 2t \wedge 3s)$ and 
    $\psi^*(17 \wedge 2t \wedge 3u) = \psi^*(1s \wedge 2u \wedge 3t)$.
    Moreover, one easily checks that
    \[
     \psi^*(1s \wedge 2t \wedge 3u) +
     \psi^*(1t \wedge 2u \wedge 3s) +
     \psi^*(1u \wedge 2s \wedge 3t) = 0.
    \]
    In particular, the image under $\psi^*$ is generated by the four
    elements 
    \begin{eqnarray*}
      \psi^*(14 \wedge 25 \wedge 36) 
      &=&\! {12} \tensor (46-56) + {13} \tensor (56-45)
      = e_{26} - e_{35};\\
      \psi^*(14 \wedge 26 \wedge 35) 
      &=&\! {12} \tensor (45-56) + {13} \tensor (56-46)
      = e_{25} - e_{36};\\
      \psi^*(15 \wedge 24 \wedge 36) 
      &=&\! {12} \tensor (56-46) + {13} \tensor (46-45)\\
      &=& -e_{26} - e_{35} + e_{36};\\
      \psi^*(15 \wedge 26 \wedge 34)
      &=&\! {12} \tensor (45-46) + {13} \tensor (46-56) \\
      &=& e_{25} - e_{26} + e_{36}.
    \end{eqnarray*}
  Since 
  \[
  \det
  \left(
  \begin{array}{rrrr}
    0 & 1 & -1 & 0 \\
    1 & 0 & 0 & -1 \\
    0 & -1 & -1 & 1 \\
    1 & -1 & 0 & 1 
  \end{array}
  \right) = 3,
  \]
  it follows that $\tilde{H}_1(\M[7];\Z) \cong \Z_3$.
  Moreover, by Lemma~\ref{generators-lem} and symmetry, $\gamma_6 =
  (12-23)
  \wedge (45-56)$ must be a generator of $\tilde{H}_1(\M[7];\Z)$.

  For $r > 2$, assume by induction that
  $\tilde{H}_{r-2}(\M[3r-2];\Z)$ is a group of order three.
  Again by Lemma~\ref{generators-lem},
  this group is generated by any element of the form
  $\pi(\gamma_{3r-3})$, where 
  $\gamma_{3r-3}$ is defined as in (\ref{gamman-eq})  
  and $\pi \in \Symm{3r-2}$.
  Flipping $\pi(1)$ and $\pi(3)$ yields $-\pi(\gamma_{3r-3})$; hence
  the action of $\Symm{3r-2}$ on $\tilde{H}_{r-2}(\M[3r-2];\Z)$ is
  given by $\pi(z) = \sgn(\pi)\cdot z$.

  By induction, we have the following exact sequence:
  \[
  \begin{CD}
    R_{r-3}^{3r-5}  
    @>\psi^*>>  
    \langle 12 \rangle \tensor \Z_3 \oplus \langle 13 \rangle \tensor
    \Z_3
    @>\iota^*\circ \varphi^*>>  
    \tilde{H}_{r-1}(\M[3r+1])
    @>>> 0.
  \end{CD}
  \]
  Another application of 
  Theorem~\ref{exseq0356-thm} and
    Corollary~\ref{exseq0356-cor} yields that 
  \begin{eqnarray*}
  \psi^*(1s \wedge 2t \wedge 3u \tensor z) &=& 
  {12} \tensor (su-tu) \wedge z + {13}  \tensor (tu-st)
  \wedge z \\
  &=:& {12} \tensor \delta + {13}  \tensor \delta'.
  \end{eqnarray*}
  Note that $\delta = (s,t,u)(\delta') = \delta'$, 
  which implies that the image under $\psi^*$ is contained in 
  $({12}+{13})\tensor \Z_3$. Moreover,
  \[
  \psi^*(14\wedge 25\wedge 36 \tensor \gamma_{3r-6}^{(6)})  
  = 12 \tensor (46-56) \wedge \gamma_{3r-6}^{(6)}
  + 13 \tensor (56-45) \wedge \gamma_{3r-6}^{(6)},
  \]
  where $\gamma_{3r-6}^{(6)}$ is defined as in 
  (\ref{gamman-eq}) but with all
  elements shifted six steps up. This is nonzero; hence
  the image under $\psi^*$ is indeed equal to $({12}+{13})\tensor
  \Z_3$. We conclude 
  that $\tilde{H}_{r-1}(\M[3r+1];\Z) \cong \Z_3$.
  \end{proof}

  \subsection{The general case}
  \label{generalcase-sec}
    
  Bouc \cite{Bouc} proved that the exponent of the group 
  $\tilde{H}_{\nu_n}(\M[n]; \Z)$ divides 
  nine whenever $n = 3r + 3$ for some $r \ge 3$. Using the exact 0-3-4
  sequence in   
  Section~\ref{exseq034-sec}, {\shareshian} and Wachs
  extended and improved this result:
  \begin{theorem}[{\shareshian} and Wachs \cite{ShWa}]
    \label{ShWa-thm}
    For $n \in \{7,10,12,13\}$ and for $n
    \ge 15$, $\tilde{H}_{\nu_n}(\M[n]; \Z)$ is
    of the form $(\Z_3)^{e_n}$ for some $e_n \ge 1$.
    The torsion subgroup of $\tilde{H}_{\nu_n}(\M[n]; \Z)$
    is again an elementary $3$-group for $n \in \{9,11\}$ and  
    zero for $n \in \{1,2,3,4,5,6,8\}$. For the remaining case $n=14$, 
    $\tilde{H}_{\nu_n}(\M[n]; \Z)$ is a finite group with
    nonvanishing $3$-torsion. 
  \end{theorem}
  The only existing proofs for the cases $n \in \{9,11,12\}$ are
  computer-based. Our hope is that one may exploit
  properties of the exact sequences in this paper to find a proof
  without computer assistance. 

  By Theorem~\ref{bouctor-thm}, $e_{3r+1} = 1$ whenever
  $r \ge 2$. 
  In Section~\ref{highermatch-sec}, we show that
  $e_{3r+3}$ is bounded by a polynomial of degree $3$ and that
  $e_{3r+5}$ is bounded by a polynomial of degree $6$.

  \begin{corollary}
    For $n = 1$ and for $n \ge 3$, the  group 
    $\tilde{H}_{\nu_n}(\M[n]; \Z)$ is nonzero.
    In particular, the connectivity degree of $\M[n]$ 
    equals  $\nu_n-1$. 
    \label{conndepth-cor}
  \end{corollary}

  For $n = 14$, the following is known:
  \begin{theorem}[Jonsson \cite{m14}]
    $\tilde{H}_{4}(\M[14]; \Z)$ is a finite nontrivial group of
    exponent a multiple of $15$.
    \label{m14-thm}
  \end{theorem}

  \section{Higher-degree homology}
  \label{generalmatch-sec}

  In Section~\ref{highermatch-sec}, we detect $3$-torsion in
  higher-degree homology groups of $\M[n]$. 
  In Section~\ref{pnot3-sec}, we 
  demonstrate that whenever the degree falls within a given interval,
  the whole homology group is a $3$-group.
  We discuss the situation outside this interval in
  Section~\ref{further5-sec}, providing some loose evidence for the
  existence of large intervals with $5$-torsion.
  In Section~\ref{boundsmatch-sec}, we proceed with
  upper bounds on the dimension of the homology over $\Z_3$.

  \subsection{$3$-torsion in higher-degree homology groups}
  \label{highermatch-sec}

  This section builds on work previously published in the author's thesis
  \cite{thesis}. 
  First, let us state an elementary but useful result; the proof is
  straightforward.
  \begin{lemma}
    Let $k \ge 1$ and let $G$ be a graph on $2k$ vertices. Then
    $\M(G)$ admits a collapse to a complex of dimension at most
    $k-2$.
    \label{evencoll-lem}
  \end{lemma}

  Let $k_0 \ge 0$ and let $\mathcal{G} = \{ G_k : k \ge k_0\}$  
  be a family of graphs such that the following conditions hold:
  \begin{itemize}
  \item
    For each $k \ge k_0$, the vertex set of $G_k$ is $[2k+1]$.
  \item 
    For each $k > k_0$ and for each vertex $s$ such that $1s$ is an
    edge in $G_k$, the induced subgraph $G_k([2k+1]
    \setminus \{1,s\})$ is
    isomorphic to $G_{k-1}$.
  \end{itemize}
  We say that such a family is {\em compatible}.
  \begin{proposition}
    In each of the following three cases, $\mathcal{G} =
    \{G_k :  k \ge k_0\}$ is a compatible family: 
    \begin{itemize}
    \item[(1)]
      $G_{k} = K_{2k+1}$ for all $k$.
    \item[(2)]
      $G_{k} = K_{k+1,k}$ for all $k$, where
      $K_{k+1,k}$ is the complete bipartite graph with blocks $[k+1]$
      and 
      $[k+2, 2k+1]$.
    \item[(3)]
      $G_{k} = K_{2k+1} \setminus \{23, 45, 67,
      \ldots, 2k(2k+1)\}$ for all $k$.
    \end{itemize}
    \label{compatible-prop}
  \end{proposition}
  \begin{proof}
    It suffices to prove that $G_{k}([2k+1] \setminus \{1,s\})$ 
    is isomorphic to 
    $G_{k-1}$ whenever $1s$ is an edge in $G_{k}$ and $k > k_0$. 
    This is immediate in all three cases.
  \end{proof}
  
  Now, fix $k_0,n,d \ge 0$.
  Let $\mathcal{G} = \{G_{k} :  k \ge k_0\}$ be a family of
  compatible graphs and let $\gamma$ be an element
  in $\tilde{H}_{d-1}(\M[n];\Z)$, hence a cycle of type
  $\brackom{n}{d}$. For each $k \ge k_0$, define a map
  \[
  \left\{
  \begin{array}{l}
    \theta_k : \tilde{H}_{k-1}(\M(G_k);\Z)
    \rightarrow \tilde{H}_{k-1+d}(\M[2k+1+n];\Z) \\
    \theta_k(z) =  z \wedge \gamma^{(2k+1)},
  \end{array}
  \right.
  \]
  where we obtain $\gamma^{(2k+1)}$ from
  $\gamma$ by replacing each occurrence of the vertex $i$ with 
  $i+2k+1$ for every $i \in [n]$. Note that 
  $\tilde{H}_{k-1}(\M(G_k);\Z)$ is the top homology group of
  $\M(G_k)$ (provided $G_k$ contains matchings of size $k$).
  For any prime $p$, we have that
  $\theta_k$ induces a homomorphism
  \[
  \theta_{k} \tensor_\Z \iota_p : \tilde{H}_{k-1}(\M(G_{k});\Z)
  \tensor_\Z \Z_p \rightarrow
  \tilde{H}_{k-1+d}(\M[2k+1+n];\Z) \tensor_\Z
  \Z_p,
  \]
  where $\iota_p : \Z_p \rightarrow \Z_p$ is the identity.
  \begin{theorem}
    With notation and assumptions as above, 
    if $\theta_{k_0} \tensor_\Z \iota_p$ is a monomorphism, 
    then $\theta_{k} \tensor_\Z \iota_p$ is a monomorphism for each
    $k \ge k_0$.
    If, in addition, the exponent of $\gamma$ 
    in $\tilde{H}_{d-1}(\M[n];\Z)$ is $p$, then
    we have a monomorphism
    \[
    \left\{
    \begin{array}{l}
      \hat{\theta}_{k} : \tilde{H}_{k-1}(\M(G_{k});\Z)
      \tensor_\Z \Z_p \rightarrow
      \tilde{H}_{k-1+d}(\M[2k+1+n];\Z) \\
      \hat{\theta}_k(z \tensor_\Z \lambda) = \theta_k(\lambda z)
      = \lambda z \wedge \gamma^{(2k+1)}
    \end{array}
    \right.
    \]
    for each $k \ge k_0$. In particular, the group
    $\tilde{H}_{k-1+d}(\M[2k+1+n];\Z)$ contains $p$-torsion of rank 
    at least the rank of $\tilde{H}_{k-1}(\M(G_{k});\Z)$.
    \label{torallovergen-thm}
  \end{theorem}
  \begin{proof}
    To prove the first part of the theorem, we use induction on $k$;
    the base case $k=k_0$ is true by assumption.
    Assume that $k > k_0$ and
    consider the head end of
    the long exact sequence 
    for the pair $(\M(G_{k}),\M(G_k \setminus \{1\}))$,
    where $G_k \setminus \{1\} = G_k([2k+1] \setminus \{1\})$:
    \[
    \begin{CD}
      & & 0 @>>> \tilde{H}_{k-1}(\M(G_k \setminus \{1\});\Z) \\
      \longrightarrow \tilde{H}_{k-1}(\M(G_{k});\Z)
      @>\hat{\omega} >> P_{k-2}(G_k) @>>>
      \tilde{H}_{k-2}(\M(G_k \setminus \{1\});\Z).
    \end{CD}
  \]
  Here, 
  \[
  P_{k-2}(G_k) = \bigoplus_{s : 1s \in G_{k}}
  \langle 1s \rangle \tensor \tilde{H}_{k-2}(\M(G_{k} \setminus
  \{1,s\});\Z)
  \]
  and  $\hat{\omega}$ is defined in the natural manner.

  Now, the group $\tilde{H}_{k-1}(\M(G_k \setminus \{1\});\Z)$ is
  zero by Lemma~\ref{evencoll-lem}.
  As a consequence, $\hat{\omega}$ is a
  monomorphism. Moreover, all groups in the second row of the above
  sequence are 
  torsion-free. Namely, the dimensions of $\M(G_k)$ 
  and $\M(G_k \setminus \{1,s\})$ are at most $k-1$ and $k-2$,
  respectively, and Lemma~\ref{evencoll-lem} 
  yields that $\M(G_k \setminus \{1\})$ is
  homotopy equivalent to a 
  complex of dimension at most $k-2$. It follows that the induced
  homomorphism 
  \[
  \hat{\omega} \tensor \iota_p :
  \tilde{H}_{k-1}(\M(G_k);\Z) \tensor \Z_p
  \rightarrow P_{k-2}(G_k) \tensor \Z_p
  \]
  remains a monomorphism.

  Now, consider the following diagram:
  \[
  \begin{CD}
    \tilde{H}_{k-1}(\M(G_k);\Z) \tensor \Z_p
    @>{\hat{\omega}\tensor \iota_p}>> P_{k-2}(G_k) \tensor \Z_p \\
    @V{\theta_k \tensor \iota_p}VV   
    @V{\theta^\oplus_{k-1}\tensor \iota_p}VV \\
    \tilde{H}_{k-1+d}(\M[2k+1+n];\Z) \tensor \Z_p
    @>{\omega \tensor \iota_p}>>
    P_{k-2+d}^{2k-1+n} \tensor \Z_p.
  \end{CD}
  \]
  Here, 
  \[
  P_{k-2+d}^{2k-1+n} = \bigoplus_{s=2}^{2k+1+n} 
  \langle 1s \rangle \tensor 
  \tilde{H}_{k-2+d}(\M[{[2,2k+1+n] \setminus \{s\}}];\Z),
  \]
  $\omega$ is defined as  
  in Theorem~\ref{exseq012-thm},
  and $\theta^\oplus_{k-1}$ is defined by
  \[
  \theta^\oplus_{k-1}(1s \tensor z)
  = 1s \tensor z \wedge \gamma^{(2k+1)}.
  \]
  One easily checks that the diagram commutes; going to the right and
  then down or going down and then to the right both give the same map
  \[
  (c_1 + \sum_{s:1s \in G_k} 1s \wedge z_{1s}) \tensor 1 \mapsto
  \sum_{s:1s \in G_k}(1s \tensor z_{1s} \wedge
  \gamma^{(2k+1)}) \tensor 1,
  \]
  where $c_1$ is a sum of oriented simplices from $\M(G_k\setminus
  \{1\})$ and each $z_{1s}$ is a sum of oriented simplices from
  $\M(G_k \setminus \{1,s\})$ satisfying $\partial(z_{1s}) = 0$ and 
  $\partial(c_1) + \sum_{s} z_{1s} = 0$.
  Moreover, $\theta^\oplus_{k-1} \tensor \iota_p$ is a
  monomorphism, because the restriction to each
  summand is a monomorphism by induction on $k$. 
  Namely, since $\mathcal{G}$ is compatible, $G_{k} \setminus \{1,s\}$
  is isomorphic to $G_{k-1}$ for each $s$ such that $1s \in G_{k}$.
  As a consequence,
  $(\theta^\oplus_{k-1} \circ \hat{\omega}) \tensor \iota_p$ is a
  monomorphism, which implies that $\theta_k \tensor \iota_p$ is a
  monomorphism.

  For the very last statement, it suffices to prove that 
  $\hat{\theta}_k$ is a well-defined homomorphism, which is true if
  and only if $\theta_k(pz) = 0$ for each $z \in
  \tilde{H}_{k-1}(\M(G_k);\Z)$. 
  Now, let $c \in \tilde{C}_{d}(\M[n];\Z)$ be such
  that $\partial(c) = p\gamma$; such a $c$ exists by assumption. We
  obtain that 
  \[
  \partial(z \wedge c^{(2k+1)}) =
  \pm z \wedge (p\gamma^{(2k+1)}) =
  \pm (pz) \wedge \gamma^{(2k+1)};
  \]
  hence $\theta_k(pz) = 0$ as desired. 
  \end{proof}
  One may generalize Theorem~\ref{torallovergen-thm} by allowing 
  a whole family $\mathcal{G}_k$ of graphs for each $k$ rather than
  just one single graph $G_k$. The condition for compatibility would
  then be that for any $G \in \mathcal{G}_k$ and for any $s$ such that
  $1s \in G$, the induced subgraph $G([2k+1] \setminus \{1,s\})$ is
  isomorphic to some graph in $\mathcal{G}_{k-1}$. We do not need this
  generalization in this paper.

  \begin{theorem}
    For $k \ge 0$ and $r \ge 4$, there is $3$-torsion of rank 
    at least $\binom{2k}{k}$ in 
    $\tilde{H}_{k-1+r}(\M[2k+1+3r];\Z)$.
    Moreover, for $k \ge 0$, there is $3$-torsion of rank 
    at least $\binom{k+1}{\lfloor (k+1)/2\rfloor}$ in 
    $\tilde{H}_{k+2}(\M[2k+10];\Z)$. To summarize, 
    $\tilde{H}_{k-1+r}(\M[2k+1+3r];\Z)$ contains nonvanishing
    $3$-torsion whenever $k \ge 0$ and $r \ge 3$.
    \label{torsionallover-thm}
  \end{theorem}
  \begin{proof}
    For the first statement, consider the compatible family 
    $\{ K_{2k+1} : k \ge 0\}$ and the cycle 
    $\gamma_{3r} \in \tilde{H}_{r-1}(\M[3r];\Z)$
    defined as in (\ref{gamman-eq}).
    By Theorem~\ref{bouctor-thm} and Lemma~\ref{generators-lem},
    \[
    \theta_{0} \tensor \iota_3 :
    \tilde{H}_{-1}(\M[1];\Z) 
    \tensor_\Z \Z_3  \cong 
    \Z \tensor_\Z \Z_3 \rightarrow
    \tilde{H}_{r-1}(\M[3r+1];\Z) \tensor_\Z \Z_3
    \]
    defines an isomorphism, where 
    $\theta_{0}(\lambda) = \lambda  \gamma_{3r}^{(1)}$.
    By Lemma~\ref{generators-lem} and 
    Theorem~\ref{ShWa-thm}, $\gamma_{3r}$ has exponent $3$ 
    in $\tilde{H}_{r-1}(\M[3r];\Z)$; hence
    Theorem~\ref{torallovergen-thm} yields that the group
    $\tilde{H}_{k-1+r}(\M[2k+1+3r];\Z)$ contains 
    $3$-torsion of rank at least the rank of 
    the group $\tilde{H}_{k-1}(\M[2k+1];\Z)$. 
    By a result due to 
    Bouc \cite{Bouc}, this rank equals $\binom{2k}{k}$.

    For the second statement, consider the compatible family
    $\{G_k = K_{2k+1} \setminus \{23, 45, 67, \ldots, 2k(2k+1)\} : k
    \ge 
    1\}$ and the cycle $\gamma_6 = (12-23) \wedge (45-56)  
    \in \tilde{H}_1(\M[7];\Z)$.
    For $k=1$, we obtain that $G_1$ is the graph $P_3$ on three
    vertices with edge
    set $\{12,13\}$; clearly, $\tilde{H}_0(\M(P_3);\Z)
    \cong \Z$. As a consequence, 
    \[
    \theta_{1} \tensor \iota_3 :
    \tilde{H}_{0}(\M(P_3);\Z) 
    \tensor_\Z \Z_3 \rightarrow
    \tilde{H}_{2}(\M[10];\Z) \tensor_\Z \Z_3
    \]
    is an isomorphism; apply 
    Theorem~\ref{bouctor-thm}. Proceeding as in the first 
    case and using the fact that $\gamma_6$ has exponent 
    $3$ in $\tilde{H}_{1}(\M[7];\Z)$, we conclude that 
    $\tilde{H}_{k+1}(\M[2k+8];\Z)$ contains 
    $3$-torsion of rank at least the rank of 
    $\tilde{H}_{k-1}(\M(G_k);\Z)$ for each $k \ge 1$. 

    It remains to show that the rank of 
    $\tilde{H}_{k-1}(\M(G_k);\Z)$ is at least
    $\binom{k}{\lfloor k/2\rfloor}$. Let $A$ be any subset of the
    removed edge set 
    \[
    E = \{23, 45, \ldots, 2k(2k+1)\}
    \]
    such that the size of $A$ is $\lfloor k/2 \rfloor$; write $B = E
    \setminus A$. Consider the complete bipartite
    graph $G_k^A$ with one block equal to $\{1\} \cup \bigcup_{e
    \in A} e$ and the other block equal to $\bigcup_{e \in B} e$.
    For even $k$, the size of the
    ``$A$'' block is $k+1$; for odd $k$, the size of the ``$A$'' block
    is 
    $k$. It is clear that $G_k^A$ is a subgraph of $G_k$. 

    Label the vertices in $[2,2k+1]$ as $s_1, t_1, s_2, t_2, \ldots,
    s_{k}, t_{k}$ such that $s_it_i \in A$ for even $i$ and 
    $s_it_i \in B$ for odd $i$. Consider the matching 
    \[
    \sigma_A = \{1s_1, t_1s_2, t_2s_3, \ldots, t_{k-1}s_k\}.
    \]
    One easily checks that 
    $\sigma_A \in \M(G_k^{A'})$ if and only if $A = A'$. 
    Now, as observed by Shareshian and Wachs \cite[(6.2)]{ShWa},
    $\M(G^A_k)$ is an orientable pseudomanifold. Defining $z_A$ to be
    the fundamental cycle of $\M(G^A_k)$, we obtain that $\{z_A : A
    \subset E, \# A =  \lfloor k/2 \rfloor\}$ forms an independent set
    in $\tilde{H}_{k-1}(\M(G_k);\Z)$, which concludes the proof.
  \end{proof}

    \ \\
    Let $G_k = K_{2k+1} \setminus \{23, 45, 67, \ldots, 2k(2k+1)\}$ be
    the graph in the above proof.
    Based on computer calculations for $k \le 5$, we conjecture that 
    the rank $r_k$ of $\tilde{H}_{k-1}(\M(G_k);\Z)$ equals
    the coefficient of $x^k$ in $(1+x+x^2)^k$; this is sequence
    A002426 in Sloane's Encyclopedia \cite{EIS}. Equivalently,
    \[
    \sum_{k \ge 0} r_k x^k = \frac{1}{\sqrt{1-2x-3x^2}}.
    \]

  \begin{proposition}[Jonsson \cite{m14}]
    We have that $\tilde{H}_{4}(\M[13];\Z) \cong T \oplus
    \Z^{24596}$, where $T$ is a finite group containing 
    $\Z_3^{10}$ as a subgroup.
    \label{m13-prop}
  \end{proposition}
  \begin{corollary}
    For $n \ge 1$, there is nonvanishing $3$-torsion
    in the homology group $\tilde{H}_d(\M[n]; \Z)
    = \tilde{H}_{k-1+r}(\M[2k+1+3r];\Z)$ whenever
    \[
    \left\lceil\frac{n-4}{3}\right\rceil \le d \le 
    \left\lfloor\frac{n-6}{2}\right\rfloor
    \Longleftrightarrow
    \left\{
    \begin{array}{l}
      k \ge 0 \\
      r \ge 3
    \end{array}
    \right.   
    \]
    or $r=2$ and $k \in \{0,1,2,3\}$.
    Moreover,
    $\tilde{H}_d(\M[n]; \Z)$ is nonzero if and only
    if 
    \[
    \left\lceil\frac{n-4}{3}\right\rceil \le d \le 
    \left\lfloor\frac{n-3}{2}\right\rfloor
    \Longleftrightarrow
    \left\{
    \begin{array}{l}
      k \ge 0 \\
      r \ge 0.
    \end{array}
    \right.   
    \]
    \label{torsionallover-cor}
  \end{corollary}
  \begin{proof}
    The first statement is a consequence of
    Theorem~\ref{torsionallover-thm}, 
    Proposition~\ref{m13-prop}, and Figure~\ref{matching-fig}.
    For the second statement, Theorem~\ref{bouc-thm} yields that 
    the group $\tilde{H}_{k-1+r}(\M[2k+1+3r]; \Z)$ is infinite 
    if and only if $r \ge 0$ and $k \ge \binom{r}{2}$.
    In particular, the group is infinite
    for all $k \ge 0$ and $0 \le r \le 2$ except $(k,r) = (0,2)$.
    Since $\tilde{H}_{k-1+r}(\M[2k+1+3r]; \Z) \cong \Z_3$ when
    $k = 0$ and $r=2$, we are done by 
    Theorem~\ref{smallest-thm} and Lemma~\ref{evencoll-lem}.
  \end{proof}
  Corollary~\ref{torsionallover-cor} suggests the following
  conjecture:
  \begin{conjecture}
    The group $\tilde{H}_{d}(\M[n]; \Z)
    = \tilde{H}_{k-1+r}(\M[2k+1+3r];\Z)$ contains
    $3$-torsion if and only if 
    \[
    \left\lceil\frac{n-4}{3}\right\rceil \le d \le \left\lfloor
    \frac{n-5}{2}\right\rfloor
    \Longleftrightarrow
    \left\{
    \begin{array}{l}
      k \ge 0 \\
      r \ge 2.
    \end{array}
    \right.   
    \]
    \label{torsion-conj}
  \end{conjecture}
  By Corollary~\ref{torsionallover-cor}, the
  conjecture remains unsettled if and only if $r = 2$ and $k\ge 4$;
  for the cases $r=0$ and $r=1$, one easily checks that the homology
  is free. The conjecture would follow if we were able to settle
  Conjecture~\ref{splitexact-conj} in Section~\ref{concluding-sec}.

  \subsection{Intervals with vanishing homology over $\Z_p$ for all $p
    \neq 3$}
  \label{pnot3-sec}

  Throughout this section, let $p$ be a prime different from $3$.
  Using the exact sequences in Sections~\ref{exseq012-sec}, 
  \ref{exseq034-sec}, and \ref{exseq0356-sec}, 
  we provide bounds on $d$ and $n$ such that
  $\tilde{H}_d(\M[n];\Z_p)$ is zero.

  \begin{theorem}
    The group $\tilde{H}_d(\M[n];\Z_p) =
    \tilde{H}_{k-1+r}(\M[2k+1+3r];\Z_p)$ 
    is zero unless $2n-8 \le 5d \Longleftrightarrow 
    r \le k + 1$. 
    Moreover, for each $q \ge 0$, the following hold (notation as in
    Section~{\rm\ref{basic-sec}}): 
    \begin{itemize}
    \item
      $\tilde{H}_{2q-1}(\M[5q])$ is
      generated by cycles of type
      $\brackom{5}{2}^{q}$ and 
      $\brackom{3}{1} \wedge \brackom{5q-3}{2q-1}$.
    \item
      $\tilde{H}_{2q-1}(\M[5q+1])$ is
      generated by cycles of type $\brackom{1}{0} \wedge
      \brackom{5}{2}^{q}$. 
    \item
      $\tilde{H}_{2q}(\M[5q+3])$ is
      generated by cycles of type
      $\brackom{3}{1} \wedge \brackom{5}{2}^{q}$.
    \item
      $\tilde{H}_{2q}(\M[5q+4])$ 
      is generated by cycles of type
      $\brackom{1}{0} \wedge \brackom{3}{1} \wedge \brackom{5}{2}^q$.
    \end{itemize}
    \label{p5-thm}
  \end{theorem}

  \begin{figure}[htb]
    \begin{center}
      \begin{tabular}{|l|l||l|l|}
        \hline
        $n$ & $\mu_n$ &  $n$ & $\mu_n$ \\
        \hline
        \hline
        $5q-5$ & $2q-3$ &  $5q$ & $2q-1$ \\
        \hline
        $5q-4$ & $2q-3$ &  $5q+1$ & $2q-1$ \\
        \hline
        $5q-3$ & $2q-2$ &  $5q+2$ & $2q$ \\
        \hline
        $5q-2$ & $2q-2$ &  $5q+3$ & $2q$ \\
        \hline
        $5q-1$ & $2q-2$ &  $5q+4$ & $2q$ \\
        \hline
      \end{tabular}
    \end{center}
    \caption{$\mu_n = \lceil\frac{2n-8}{5}\rceil$ for different values of
      $n$.}
    \label{p5proof-fig}
  \end{figure}

  \begin{proof}
    Writing $\mu_n = \lceil\frac{2n-8}{5}\rceil$, we obtain 
    Figure~\ref{p5proof-fig}, which might be of some help
    when reading this proof.

    One easily checks the theorem for $n \le 5$; thus assume that $n
    \ge 6$. Assume inductively that the theorem is true for all $m \le
    n-1$. We have five cases for $n$:

    $\bullet$
    $n=5q$. The first case is perhaps the hardest. By the long exact
    0-3-4 sequence in
    Section~\ref{exseq034-sec}, we have an exact sequence
    of the form
    \[
    \begin{CD}
      \bigoplus \tilde{H}_{d-1}(\M[5q-3]) 
      @>>> 
      \tilde{H}_{d}(\M[5q])
      @>>> 
      \bigoplus \tilde{H}_{d-2}(\M[5q-4]).
    \end{CD}
    \]
    By induction, the groups on the left and right are zero whenever
    $d < 2q-1$, which
    implies that the same is true for the group in the middle.

    It remains to prove that $\tilde{H}_{2q-1}(\M[5q])$ is generated
    by cycles of type 
    $\brackom{5}{2}^{q}$ and 
    $\brackom{3}{1} \wedge \brackom{5q-3}{2q-1}$.
    For this, consider the tail end of the long exact 0-3-4 sequence: 
    \[
    \begin{CD}
      & & \displaystyle{\bigdirsum_{a,u}}\
      \langle au \rangle \tensor \tilde{H}_{2q-2}(\M[{[3,5q]\setminus \{u\}}])
      @>\varphi^*>>
      \tilde{H}_{2q-1}(\M[5q])  \\
      @>\kappa^*>>  
      \displaystyle{\bigdirsum_{s,t}}\
      \langle 1s \wedge 2t \rangle \tensor
      \tilde{H}_{2q-3}(\M[{[3,5q]\setminus \{s,t\}}])
      @>>> 0;
    \end{CD}
    \]
    see Theorem~\ref{exseq034-thm}.
    To generate $\tilde{H}_{2q-1}(\M[5q])$, we will combine two sets
    of cycles:
    \begin{itemize}
    \item[(1)]
      The first set consists of the image under $\varphi^*$ of an
      appropriate set of generators of the 
      first group in the exact sequence.
    \item[(2)]
      The second set consists of an appropriate set of cycles in the
      group $\tilde{H}_{2q-1}(\M[5q])$ such that the image under
      $\kappa^*$ of this set generates the third group in the sequence.
    \end{itemize}

    (1) By properties of $\varphi^*$, the image of any cycle 
    in the leftmost group has type $\brackom{3}{1} \wedge
    \brackom{5q-3}{2q-1}$.

    (2) Induction yields that 
    $\tilde{H}_{2q-3}(\M[{[3,5q]\setminus
        \{s,t\}}]) \cong \tilde{H}_{2q-3}(\M[5q-4])$ is
    generated by cycles of type $\brackom{1}{0} \wedge
    \brackom{5}{2}^{q-1}$.
    Now, consider a cycle $z \in \tilde{H}_{2q-3}(\M[{[3,5q]\setminus 
        \{s,t\}}])$ of type $\brackom{1}{0} \wedge
      \brackom{5}{2}^{q-1}$;
    let $x$ be the unused element in $z$ 
    corresponding to the empty cycle of type $\brackom{1}{0}$. 
    Define
    \[
    \gamma = 1s \wedge 2t + 2t \wedge sx + sx \wedge 12 + 12 \wedge tx
    + tx \wedge 1s.
    \]
    It is clear that $\kappa^*$ maps $\gamma \wedge z$ to
    $1s \wedge 2t \tensor z$ and that
    $\gamma \wedge z$ has type
    $\brackom{5}{2}^{q}$. 
    Thus we are done.

    $\bullet$
    $n=5q+1$. 
    Again using the long exact 0-3-4 sequence in
    Section~\ref{exseq034-sec}, we deduce
    that $\tilde{H}_{d}(\M[5q+1])$ is zero whenever 
    $\tilde{H}_{d-1}(\M[5q-2])$ and $\tilde{H}_{d-2}(\M[5q-3])$ are
    zero, which is true for $d < 2q-1$. 
    For $d = 2q-1$, we obtain the exact sequence
    \[
    \begin{CD}
      \displaystyle{\bigdirsum_{a,u}}\
      \langle au \rangle \tensor
      \tilde{H}_{2q-2}(\M[{[3,5q+1]\setminus \{u\}}]) 
      @>\varphi^*>>
      \tilde{H}_{2q-1}(\M[5q+1])  
      @>>> 
      0.
    \end{CD}
    \]
    By induction, 
    $\tilde{H}_{2q-2}(\M[{[3,5q+1]\setminus \{u\}}])
    \cong \tilde{H}_{2q-2}(\M[5q-2])$ is generated
    by cycles of type 
    $\brackom{3}{1} \wedge \brackom{5}{2}^{q-1}$.
    Hence properties of $\varphi^*$ yield that
    $\tilde{H}_{2q-1}(\M[5q+1])$ is generated by cycles of type
    $\brackom{3}{1} \wedge \brackom{3}{1} \wedge \brackom{5}{2}^{q-1}$.
    By the exact sequence for the pair
    $(\M[7],\M[6])$ in Section~\ref{exseq012-sec} and the 
    fact that $\tilde{H}_1(\M[7]) = 0$, we have that 
    $\tilde{H}_1(\M[6];\Z)$ is generated by cycles of type 
    $\brackom{1}{0} \wedge \brackom{5}{2}$; use
    Theorem~\ref{exseq012-thm}. As a consequence,  
    any cycle of type $\brackom{3}{1} \wedge \brackom{3}{1} \wedge
    \brackom{5}{2}^{q-1}$ can be written as a sum of cycles of 
    type $\brackom{1}{0} \wedge \brackom{5}{2} \wedge
    \brackom{5}{2}^{q-1} = \brackom{1}{0} \wedge 
    \brackom{5}{2}^{q}$.

    $\bullet$
    $n=5q+2$.
    Using the long exact 0-1-2 sequence in Section~\ref{exseq012-sec}, 
    we conclude 
    that $\tilde{H}_{d}(\M[5q+2])$ is zero whenever 
    $\tilde{H}_{d}(\M[5q+1])$ and $\tilde{H}_{d-1}(\M[5q])$ are
    zero, which is true for $d < 2q-1$. For $d = 2q-1$, we have the
    exact sequence 
    \[
    \begin{CD}
      \tilde{H}_{2q-1}(\M[{[2,5q+2]}]) @>\iota^*>>
      \tilde{H}_{2q-1}(\M[5q+2]) @>>> 0,
    \end{CD}
    \]
    where $\iota^*$ is induced by the inclusion map. By induction,
    we have that $\tilde{H}_{2q-1}(\M[{[2,5q+2]}]) \cong
    \tilde{H}_{2q-1}(\M[5q+1])$ is generated by cycles of type 
    $\brackom{1}{0} \wedge \brackom{5}{2}^q$.
    It follows that the group $\tilde{H}_{2q-1}(\M[5q+2])$ is
    generated by cycles of type 
    $\brackom{2}{0} \wedge \brackom{5}{2}^q$, which means that
    $\tilde{H}_{2q-1}(\M[5q+2]) = 0$.

    $\bullet$
    $n=5q+3$. This time, we use the long exact 0-3-5-6 sequence from
    Section~\ref{exseq0356-sec}. 
    By properties of this sequence, the group
    $\tilde{H}_{d}(\M[5q+3])$ is zero whenever 
    $\tilde{H}_{d-1}(\M[5q])$, $\tilde{H}_{d-2}(\M[5q-2])$,
    and $\tilde{H}_{d-3}(\M[5q-3])$ are zero, which is true
    for $d < 2q$. For $d=2q$, we have a surjection
    \[
    \begin{CD}
      \bigoplus
      \langle as\wedge bt \rangle \tensor
      \tilde{H}_{2q-2}(\M[{[4,5q+3] \setminus 
          \{s,t\}}]) \oplus
      \bigoplus \langle 1c \rangle \tensor
      \tilde{H}_{2q-1}(\M[{[4,5q+3]}])\\
      @V\varphi^*VV \\ 
    \tilde{H}_{2q}(\M[5q+3])
    \end{CD}
    \]
    defined as in Lemma~\ref{exseq0356-lem}.
    To establish that $\tilde{H}_{2q}(\M[5q+3])$ is generated by
    cycles of type $\brackom{3}{1} \wedge \brackom{5}{2}^q$, it
    suffices to prove that $\tilde{H}_{2q}(\M[5q+3])$ is generated by
    cycles of type $\brackom{5}{2} \wedge \brackom{5q-2}{2q-1}$.
    Namely, by induction, $\tilde{H}_{2q-2}(\M[5q-2])$ is
    generated by cycles of type $\brackom{3}{1} \wedge
    \brackom{5}{2}^{q-1}$.

    Induction yields that 
    $\tilde{H}_{2q-2}(\M[{[4,5q+3] \setminus\{s,t\}}])
    \cong \tilde{H}_{2q-2}(\M[5q-2])$ is generated
    by cycles of type
    $\brackom{3}{1} \wedge \brackom{5q-5}{2q-2}^{q-1}$
    and that $\tilde{H}_{2q-1}(\M[{[4,5q+3]}])
    \cong \tilde{H}_{2q-1}(\M[5q])$ is generated by 
    cycles of type 
    $\brackom{5}{2} \wedge \brackom{5q-5}{2q-2}$ and 
    $\brackom{3}{1} \wedge \brackom{5q-3}{2q-1}$.
    By properties of $\varphi^*$, it follows that 
    $\tilde{H}_{2q}(\M[5q+3])$ is generated by cycles of the following
    types:
    \begin{itemize}
    \item
      $\brackom{5}{2} \wedge \brackom{3}{1} \wedge
      \brackom{5q-5}{2q-2} \prec 
      \brackom{5}{2} \wedge \brackom{5q-2}{2q-1}$;
    \item
      $\brackom{3}{1} \wedge \brackom{5}{2} \wedge
      \brackom{5q-5}{2q-2} \prec \brackom{5}{2} \wedge
      \brackom{5q-2}{2q-1}$;
    \item
      $\brackom{3}{1} \wedge \brackom{3}{1} \wedge \brackom{5q-3}{2q-1}$.
    \end{itemize}
    By the discussion at the end of the case $n = 5q+1$,
    cycles of the very last type can be written as a sum of cycles of
    type $\brackom{1}{0} \wedge \brackom{5}{2} \wedge
    \brackom{5q-3}{2q-1}$. As a consequence, 
    $\tilde{H}_{2q}(\M[5q+3])$ is generated by cycles of type
    $\brackom{5}{2} \wedge \brackom{5q}{2q}$.

    $\bullet$
    $n = 5q+4$. For the final case, we again consider the 
    long exact 0-1-2 sequence from Section~\ref{exseq012-sec}. 
    We obtain that $\tilde{H}_{d}(\M[5q+4])$ is zero whenever 
    $\tilde{H}_{d}(\M[5q+3])$ and $\tilde{H}_{d-1}(\M[5q+2])$ are
    zero, which is true for $d < 2q$. 

    To conclude the proof, it remains to show that 
    $\tilde{H}_{2q}(\M[5q+4])$ is generated by cycles of 
    type $\brackom{1}{0} \wedge \brackom{3}{1} \wedge
    \brackom{5}{2}^q$. Now, induction yields that
    $\tilde{H}_{2q}(\M[5q+3])$ is generated by cycles of type
    $\brackom{3}{1} \wedge \brackom{5}{2}^{q}$.
    Hence $\tilde{H}_{2q}(\M[5q+4])$ is generated by cycles of type 
    $\brackom{1}{0} \wedge \brackom{3}{1} \wedge \brackom{5}{2}^{q}$
    as desired.
  \end{proof}

  \begin{corollary}
    If 
    \[
    \left\lceil\frac{n-4}{3}\right\rceil \le d \le 
    \left\lfloor\frac{2n-9}{5}\right\rfloor
    \Longleftrightarrow 
    0 \le k \le r-2,
    \]
    then $\tilde{H}_d(\M[n];\Z)
     = \tilde{H}_{k-1+r}(\M[2k+1+3r];\Z)$ is a nontrivial
     $3$-group.
    \label{p5-cor}
  \end{corollary}
  \begin{proof}
    This is an immediate consequence of
    Corollary~\ref{torsionallover-cor}, Theorem~\ref{p5-thm}, and 
    the universal coefficient theorem. 
  \end{proof}
  While the bottom nonvanishing groups are
  {\em elementary} $3$-groups by Theorem~\ref{ShWa-thm}, 
  we do not know whether this is true in general for the 
  groups under consideration.

  The smallest $n$ for which Corollary~\ref{p5-cor} implies something
  previously unknown is $n=22$, in which case we may conclude that
  $\tilde{H}_7(\M[22];\Z)$ is a $3$-group; note that $\nu_{22} = 6$.

  \subsection{On the existence of further $5$-torsion}
  \label{further5-sec}

  One may ask whether the upper bound $\frac{2n-9}{5}$ in 
  Corollary~\ref{p5-cor} is best possible, meaning that 
  there is $p$-torsion for some $p \neq 3$, most likely $p=5$, in
  degree $\lceil\frac{2n-8}{5}\rceil$ of the homology of $\M[n]$
  whenever the group under consideration is finite.
  Our hope is that this is indeed the case. 
  While we do not have much evidence to support this hope, we can
  provide the following potentially useful result:
  \begin{theorem}
    For each $q \ge 3$, 
    there is nonvanishing $5$-torsion in the group
    $\tilde{H}_{2q}(\M[5q+4];\Z)$
    if and only if there
    is a cycle $\gamma \in \tilde{H}_{2q}(\M[5q+4];\Z)$ 
    of type $\brackom{1}{0} \wedge \brackom{3}{1} \wedge
    \brackom{5}{2}^q$ such that $\gamma$ has exponent $5$.
    If this is the case, then there is nonvanishing $5$-torsion
    in $\tilde{H}_{2q+u}(\M[5q+4+2u];\Z)$ for each
    $u \ge 0$. 
    \label{further5-thm}
  \end{theorem}
  \begin{proof}
    The first statement is an immediate consequence of
    Theorem~\ref{p5-thm}. 

    For the second statement, 
    assume that $\gamma$ is a cycle with properties as in the
    theorem. Write $\gamma = \gamma_5 \wedge \gamma'$, where
    $\gamma_5$ is of type $\brackom{5}{2}$ and
    $\gamma'$ is of type $\brackom{1}{0} \wedge \brackom{3}{1}
    \wedge \brackom{5}{2}^{q-1}$. 
    It is clear that the exponent of $\gamma'$ in
    $\tilde{H}_{2q-2}(\M[5q-1];\Z)$ is a finite multiple of $5$.  
    Namely, $\gamma'$ is of type $\brackom{14}{5}\wedge
    \brackom{5q-15}{2q-6}$ and $\gamma$ has exponent $5$.

    Now, consider the compatible family $\mathcal{G} = \{K_{k+1,k} : k
    \ge 2\}$; 
    recall Proposition~\ref{compatible-prop}. 
    We claim that 
    every element $z \in \tilde{H}_1(\M[5];\Z)$ has the property
    that $2z$ is a sum of cycles, each having the form
    \[
    ac \wedge bd + bd \wedge ae + ae \wedge bc +
    bc \wedge ad + ad \wedge be + be \wedge ac,
    \]
    where $\{a,b,c,d,e\} = [5]$;
    this is the fundamental cycle of $\M[G_{a,b}]$, where
    $G_{a,b}$ is the complete bipartite graph with blocks
    $\{a,b\}$ and $\{c,d,e\}$. To prove the claim, let $T$ be the
    subgroup of $\tilde{H}_1(\M[5];\Z)$ 
    generated by the fundamental cycles of
    $G_{1,2}$, $G_{2,3}$, $G_{3,4}$, $G_{4,5}$,$G_{5,1}$, and
    $G_{1,3}$.
    One easily checks that the matrix of the natural projection from
    $T$ to the group 
    generated by $51 \wedge 23$, $12 \wedge 34$, $23 \wedge 45$, $34
    \wedge 51$, $45 \wedge 12$, and $13 \wedge 24$ has determinant
    $\pm 2$.
    Since $\tilde{H}_1(\M[5];\Z) \cong \Z^6$, the claim is settled.

    As a consequence, 
    we may assume that $\gamma_5$ is the fundamental cycle of
    $\M(K_{3,2})$. In particular, the map
    \[
    \theta_{2} \tensor \iota_5 :
    \tilde{H}_{1}(\M(K_{3,2});\Z) 
    \tensor_\Z \Z_5 \rightarrow
    \tilde{H}_{2q}(\M[5q+4];\Z) \tensor_\Z \Z_5
    \]
    defined by $\theta_2(z) = z \wedge \gamma'$
    is a monomorphism; $\gamma = \gamma_5 \wedge \gamma'$.
    Now, applying Theorem~\ref{torallovergen-thm}, 
    we deduce that we have a monomorphism
    \[
    \theta_{2+u} \tensor \iota_5 :
    \tilde{H}_{1+u}(\M(K_{3+u,2+u});\Z) 
    \tensor_\Z \Z_5 \rightarrow
    \tilde{H}_{2q+u}(\M[5q+4+2u];\Z) \tensor_\Z \Z_5
    \]
    defined by $\theta_{2+u}(z) = z \wedge (\gamma')^{(2u)}$ for each
    $u \ge 0$; notation is as in Section~\ref{highermatch-sec}.
    Since the exponent of $\gamma'$ in
    $\tilde{H}_{2q-2}(\M[5q-1];\Z)$  
    is a finite multiple of $5$, there is indeed
    nonvanishing  
    $5$-torsion in $\tilde{H}_{2q+u}(\M[5q+4+2u];\Z)$.
  \end{proof}

  \begin{corollary}
    If there is nonvanishing $5$-torsion in 
    $\tilde{H}_{2q}(\M[5q+4];\Z)$ for each $q \ge 3$, then
    $\tilde{H}_{d}(\M[n];\Z) = \tilde{H}_{k-1+r}(\M[2k+1+3r];\Z)$
    contains nonvanishing $5$-torsion whenever 
    \[
    \left\lceil \frac{2n-8}{5}\right\rceil \le 
    d \le \left\lfloor \frac{n-7}{2}\right\rfloor
    \Longleftrightarrow
    4 \le r \le k+1.
    \]
    \label{further5-cor}
  \end{corollary}

  \subsection{Bounds on the homology over $\Z_3$}
  \label{boundsmatch-sec}

  The goal of this section is to provide nontrivial upper bounds on
  the dimension of $\tilde{H}_d(\M[n];\Z_3)$ when 
  $n$ and $d$ satisfy the conditions in Corollary~\ref{p5-cor}.
  To achieve this,
  we use the long exact 0-$e$-2 sequence from Section~\ref{exseq0e2-sec}
  and the long exact 0-2-3-5 sequence
  from Section~\ref{exseq0235-sec}. 

  Define
  \[
  \left\{
  \begin{array}{rcl}
    \betti{n}{d} &=& \dim_{\Z_3} \tilde{H}_d(\M[n];\Z_3)
    \\[1ex]
    \abetti{n}{d} &=&  \dim_{\Z_3} \tilde{H}_d(\M[n] \setminus
    12;\Z_3).
  \end{array}
  \right.
  \]
  \begin{lemma}
    For all $n \ge 2$ and all $d$, we have that
    \[
      \betti{n}{d} \le \abetti{n}{d} + \betti{n-2}{d-1}.
    \]
    For $n \ge 5$ and all $d$, we have that
    \[
      \abetti{n}{d} \le \betti{n-3}{d-1} + 
      2\mbox{$\binom{n-3}{2}$} \betti{n-5}{d-2} + 
      (n-3) \abetti{n-2}{d-1}.
    \]
    \label{bettiineq-lem}
  \end{lemma}
  \begin{proof}
    The inequalities are immediate consequences of
    Theorems~\ref{exseq0e2-thm} and \ref{exseq0235-thm}.
  \end{proof}

  Define
  $\hat{\beta}_{k,r} = \betti{n}{d}$ and
  $\hat{\alpha}_{k,r} = \abetti{n}{d}$,
  where $k$ and $r$ are defined as in (\ref{nd2kr-eq}).
  \begin{corollary}
    For $k \ge 0$, $r \ge 0$, and $k+r \ge 1$, we have that
    \begin{eqnarray*}
    \hat{\beta}_{k,r} &\le& \hat{\alpha}_{k,r} +
    \hat{\beta}_{k-1,r};
    \\
    \hat{\alpha}_{k,r} &\le& \hat{\beta}_{k,r-1} +
    2\mbox{$\binom{2k+3r-2}{2}$}\hat{\beta}_{k-1,r-1}
    + (2k+3r-2)\hat{\alpha}_{k-1,r}.
    \end{eqnarray*}
    \label{bettiineq-cor}
  \end{corollary}

  \begin{theorem}
    For each $k \ge 0$, there are polynomials $f_k(r)$ and $g_k(r)$ of
    degree $3k$ with dominating term $\frac{3^k}{k!}r^{3k}$ such that 
    \[
    \left\{
    \begin{array}{rcl}
      \hat{\beta}_{k,r} &\le& f_{k}(r)
      \\[1ex] 
      \hat{\alpha}_{k,r} &\le& g_{k}(r)
    \end{array}
    \right.
    \]
    for all $r \ge k+2$.
    Equivalently, 
    \[
    \left\{
    \begin{array}{rcl}
      \betti{n}{d} &\le& f_{3d-n+4}(n-2d-3)
      \\[1ex]
      \abetti{n}{d} &\le& g_{3d-n+4}(n-2d-3) 
    \end{array}
    \right.
    \]
    for all $n \ge 7$ and $\lceil\frac{n-4}{3}\rceil \le d \le
    \lfloor\frac{2n-9}{5}\rfloor$.
    \label{betti-thm}
  \end{theorem}
  \begin{proof}
    For $k=0$, we have that $\hat{\beta}_{0,r} = 1$ for all 
    $r \ge 2$;
    use Theorem~\ref{bouctor-thm}. It is known that
    $\hat{\alpha}_{0,2} \le 1$ \cite[Th. 11.20]{thesis}; indeed, it is
    not hard to prove that
    $\tilde{H}_1(\M[7] \setminus e;\Z) \cong \tilde{H}_1(\M[7]
    \setminus e;\Z_3) \cong \Z_3$. 
    Moreover, 
    Lemma~\ref{bettiineq-lem} implies that $1 = \hat{\beta}_{0,r} \le
    \hat{\alpha}_{0,r} \le \hat{\beta}_{0,r-1} = 1$ for $r \ge 3$.

    Assume that $k \ge 1$ and $r \ge k+3$.
    By Corollary~\ref{bettiineq-cor} and induction on $k$, we obtain
    that
    \begin{eqnarray*}
    \hat{\alpha}_{k,r} &\le& \hat{\beta}_{k,r-1} +
    2\mbox{$\binom{2k+3r-2}{2}$}f_{k-1}(r-1)
    + (2k+3r-2)g_{k-1}(r);\\
    \hat{\beta}_{k,r} &\le& \hat{\alpha}_{k,r} +
    f_{k-1}(r),
    \end{eqnarray*}
    where $f_{k-1}$  and $g_{k-1}(r)$ are polynomials with properties
    as in the theorem. 
    As a consequence, 
    \[
    \hat{\beta}_{k,r} - \hat{\beta}_{k,r-1} 
    \le 
    2\mbox{$\binom{2k+3r-2}{2}$}f_{k-1}(r-1)
    + (2k+3r-2)g_{k-1}(r) + f_{k-1}(r).
    \]
    Now, the right-hand side is of the form
    \[
    h_k(r) = (3r)^2 \cdot \frac{3^{k-1}r^{3k-3}}{(k-1)!} + \rho_k(r) 
    =  \frac{3^{k+1}r^{3k-1}}{(k-1)!} + \rho_k(r),
    \]
    where $\rho_k(r)$ is a polynomial of degree at most $3k-2$.
    Summing over $r$, we obtain that
    \[
    \hat{\beta}_{k,r} \le 
    \hat{\beta}_{k,k+2} +  \sum_{i=k+3}^r h_k(r).
    \]
    The right-hand side is a polynomial $f_k(r)$ in $r$ with
    dominating term
    \[
    \frac{3^{k+1}}{(k-1)!}\cdot \frac{r^{3k}}{3k}
    = 
    \frac{3^k r^{3k}}{k!}.
    \]
    Defining 
    \[
    g_k(r) = f_k(r-1) + 
    2\mbox{$\binom{2k+3r-2}{2}$}f_{k-1}(r-1)
    + (2k+3r-2)g_{k-1}(r),
    \]
    we obtain a bound on $\hat{\alpha}_{k,r}$ with similar properties,
    which concludes the proof.
  \end{proof}
  For $k \ge 1$, one may extend the theorem to all $r \ge 0$ by adding
  a sufficiently large constant to each of $f_k(r)$ and $g_k(r)$.

  Let us provide a more precise bound for the case $k=1$.
  \begin{theorem}
    We have that
    $\betti{3}{0} = 2$, 
    $\betti{6}{1} = 16$, $\betti{9}{2} = 50$, 
    $\betti{12}{3} = 56$, and
    \[
    \betti{3r+3}{r}
    \le \frac{6 r^3  + 9 r^2 + 5 r}{2} - 73
    \]
    for $r \ge 4$.
    \label{bettik1-thm}
  \end{theorem}
  \begin{proof}
    With notation as in the proof of 
    Theorem~\ref{betti-thm}, Lemma~\ref{bettiineq-lem} implies that
    \[
    \hat{\beta}_{1,r} \le \hat{\beta}_{1,r-1} +
    2\mbox{$\binom{3r}{2}$} + 3r + 1
    = \hat{\beta}_{1,r-1} + 9r^2+1.
    \]
    Figure~\ref{matching-fig} and a straightforward computation yield
    the theorem.
  \end{proof}
  The first few values on the bound in Theorem~\ref{bettik1-thm},
  starting with $r=4$, are 201, 
  427, 752, 1194, and 1771.

  The set of pairs $(n,d)$ corresponding to a
  given $k$ in Theorem~\ref{betti-thm} is of the form $\{ v+ rw : r
  \ge k+2\}$,
  where $v = (2k+1,k-1)$ and $w = (3,1)$.
  Choosing other vectors $v$ and $w$, we obtain other
  sequences of Betti numbers. In this more general situation, it might
  be of interest to study other fields than $\Z_3$. 
  For $w = (2,1)$ and any field, the growth is 
  at least exponential as soon as $v = (n_0,d_0)$ for some $n_0$ and
  $d_0$  satisfying $n_0 \ge 2d_0+3$. Namely, 
  over $\Q$, $\betti{n_0+2q}{d_0+q}$ is known to equal the
  number of self-conjugate standard Young tableaux of size 
  $n_0+2q$ with a Durfee square of size $n_0-2d_0-2$ \cite{Bouc}. One
  easily checks that the number of such tableaux grows at least
  exponentially when $q$ tends to infinity.
  Yet, if we were to pick a vector $w = (a,b)$ such that $a/b > 2$,
  then the rational homology would disappear for sufficiently
  large $q$; apply Theorem~\ref{bouc-thm}. 
  
  By Theorem~\ref{torsionallover-thm}, there is 
  $3$-torsion of rank at least $\binom{2k}{k}$ in
  the group $\tilde{H}_{k-1+r}(\M[2k+1+3r];\Z)$ for 
  $k \ge 0$ and $r \ge 4$.
  As a consequence, for $\Z_3$, the
  growth is at least exponential for every
  $(a,b)$ satisfying $2 \le a/b<3$.
  Namely, writing $k_0 = 3d_0-n_0+4$ and $\delta = 3b-a$ and
  assuming that $2 < a/b < 3$, we have that 
  \[
  \betti{n_0+aq}{d_0+bq}
  = 
  \hat{\beta}_{k_0 + q \delta,n_0-2d_0+q(a-2b)-3}
  \ge 
  \binom{2(k_0 + q \delta)}{k_0 + q \delta}
  \]
  as soon as $n_0-2d_0+q(a-2b) \ge 7$.

  Finally, let us consider $\Z_p$ for $p \neq 3$. By
  Theorem~\ref{p5-thm}, whenever
  $a/b > 5/2$, we have that $\betti{n_0+aq}{d_0 + bq}$ is zero over
  $\Z_p$ for sufficiently large $q$. The situation 
  remains unclear for $2 < a/b \le 5/2$.

  \section{Concluding remarks and open problems}
  \label{concluding-sec}

  From our viewpoint, the most important open problem regarding the
  homology of $\M[n]$ is whether there
  exists other torsion than $3$-torsion for $n \neq 14$.
  In light of the discussion in Section~\ref{further5-sec}, 
  we are tempted to conjecture the following:
  \begin{conjecture}
    $\tilde{H}_{d}(\M[n];\Z) = \tilde{H}_{k-1+r}(\M[2k+3r+1];\Z)$
    contains nonvanishing $5$-torsion whenever
    \[
    \left\lceil\frac{2n-8}{5}\right\rceil \le d \le 
    \left\lfloor\frac{n-6}{2}\right\rfloor
    \Longleftrightarrow
    3 \le r \le k+1.
    \]
  \end{conjecture}
  The bounds are exactly the same as in 
  Corollary~\ref{further5-cor}, except that 
  the upper bound in the corollary is
  $\lfloor\frac{n-7}{2}\rfloor$ rather than
  $\lfloor\frac{n-6}{2}\rfloor$. 
  In fact, the conjecture would be true for $d =
  \frac{n-6}{2}$ and all even $n \ge 14$ if the following conjecture
  were true: 

  \begin{figure}[htb]
    \begin{center}
        \begin{tabular}{|r||c|c|c|c|c|c|}
          \hline
          $\tilde{H}_i(\M[n]\setminus e;\Z)$
          &  $i=-1$ & \ 0 \ & \ 1 \ & \ 2 \ & \ 3 \ & \ 4 \
          \\
          \hline
          \hline
          $n = 2$ & $\Z$ & - & - & - & - & - \\
          \hline
          $3$ & - & $\Z$ & - & - & - & -  \\
          \hline
          $4$ & - & $\Z^2$ & - & - & - & -  \\
          \hline
          $5$ & - & - & $\Z^4$ & - & - & -  \\
          \hline
          $6$ & - & - & $\Z^{14}$ & - & - & - \\
          \hline
          $7$ & - & - & $\Z_3$ & $\Z^{14}$ & - & -  \\
          \hline
          $8$ & - & - & - & $\Z^{116}$ & - & -  \\
          \hline
          $9$  & - & - & - & $\Z_3^7 \oplus \Z^{42}$ &
          $\Z^{50}$ - & -  \\
          \hline
          $10$ & - & - & - & $\Z_3$ &
          $\Z^{1084}$ - & -  \\
          \hline
          $11$  & - & - & - & - & $\Z_3^{37} \oplus
          \Z^{1146}$ & $\Z^{182}$  \\
          \hline
        \end{tabular}
    \end{center}
    \caption{The homology of $\M[n] \setminus e$ for $n \le 11$.} 
    \label{matching2-fig}
  \end{figure}

  \begin{conjecture}
    The sequence
    \[
    \begin{CD}
      0 \longrightarrow
      \tilde{H}_{d}(\M[n]\setminus e; \Z) \longrightarrow
      \tilde{H}_{d}(\M[n]; \Z) \longrightarrow
      \langle e \rangle \tensor \tilde{H}_{d-1}(\M[{[n]\setminus e}]; \Z)  
      \longrightarrow 0,
    \end{CD}
    \]
    cut from the long exact 0-$e$-2 sequence in
    Section~\ref{exseq0e2-sec}, is split exact for every $n \ge 3$ and
    every $d$.
    \label{splitexact-conj}
  \end{conjecture}
  We have checked the conjecture up to $n=11$ using
  computer; see  
  Figure~\ref{matching2-fig} and compare to Figure~\ref{matching-fig}.
  If Conjecture~\ref{splitexact-conj} were true for all $n$, then
  we would have $p$-torsion in $\tilde{H}_{d+k}(\M[n+2k];\Z)$ 
  for all $k \ge 0$ whenever $\tilde{H}_{d}(\M[n];\Z)$ contains 
  $p$-torsion.

  Define $\hat{\beta}_{k,r} = \dim_{\Z_3}
  \tilde{H}_{k-1+r}(\M[2k+1+3r];\Z_3)$.
  Conjecture~\ref{splitexact-conj} being true for the coefficient ring
  $\Z_3$ would imply that $\hat{\beta}_{k-1,r} \le
  \hat{\beta}_{k,r}$. Combined with a quite modest conjecture about 
  the behavior of $\{\hat{\beta}_{k,r} : r \ge 1\}$ for each fixed
  $k$, this would yield nontrivial lower bounds on $\hat{\beta}_{k,r}$
  for every $k,r \ge 0$:
  \begin{proposition}
    Suppose that $\hat{\beta}_{k-1,r} \le \hat{\beta}_{k,r}$ for 
    all $k \ge 1$ and $r \ge 0$. Suppose further that there are
    positive numbers  
    $\{C_k : k \ge 0\}$ such that 
    $C_k\hat{\beta}_{k,r} \ge \hat{\beta}_{k,r-1}$ for 
    all $k \ge 0$ and $r \ge 1$. Then
    $\hat{\beta}_{k,r}$ is bounded from below by a polynomial of
    degree $k$.
    \label{betticonj-prop}
  \end{proposition}
  \begin{proof}
    By the long exact 0-1-2 sequence, we have that
    \[
    (2k+3r)\hat{\beta}_{k-1,r-1} \le 
    \hat{\beta}_{k,r-1} + \hat{\beta}_{k-2,r}
    \]
    for $r \ge 1$.
    Applying our assumptions, we obtain that 
    \[
    (2k+3r)\hat{\beta}_{k-1,r-1} \le 
    C_k \hat{\beta}_{k,r} + \hat{\beta}_{k,r}
    = (C_k + 1)\hat{\beta}_{k,r},
    \]
    which yields that
    $\hat{\beta}_{k,r} \ge (C_k +
    1)^{-1}(2k+3r)\hat{\beta}_{k-1,r-1}$.
  \end{proof}
  For $k \le 2$, $\hat{\beta}_{k,r}$ is indeed bounded from below by a
  polynomial of degree $k$ \cite{ShWa}.

  \begin{figure}[htb]
    \begin{center}
        \begin{tabular}{|r||l|l|l|l|l|l|}
          \hline
          Exponents
          &  $k=0$ & \ 1 \ & \ 2 \ & \ 3 \ & \ 4 \ & \ 5 \
          \\
          \hline
          \hline
          $r = 0$ & $\infty$ & $\infty$ & $\infty$ & $\infty$ &
          $\infty$ & $\infty$ \\
          \hline
          $1$ & $\infty$ & $\infty$ & $\infty$ & $\infty$ &
          $\infty$ &  $\infty$ \\ 
          \hline
          $2$ & $3$ & $\infty, 3$ & $\infty, 3$ & $\infty, 3^*, e?$ &
          $\infty, e?$ & $\infty, e?$ \\  
          \hline
          $3$  & $3$ & $3$ & $3^*, 5^*, e?$ &
           $\infty, 3^*, e?$ & $\infty, 3^*, e?$ & $\infty, 3^*, e?$ \\
          \hline
          $4$  & $3$ & $3$ & $3$ &
          $3^*, e?$ & $3^*, e?$ & $3^*, e?$ \\
          \hline
          $5$  & $3$ & $3$ & $3$ &
          $3^*$ & $3^*, e?$ & $3^*, e?$ \\
          \hline
          $6$  & $3$ & $3$ & $3$ &
          $3^*$ & $3^*$ & $3^*, e?$ \\
          \hline
          $7$  & $3$ & $3$ & $3$ &
          $3^*$ & $3^*$ & $3^*$ \\
          \hline
        \end{tabular}
    \end{center}
    \caption{List of all possible infinite and prime power exponents
      of elements in $\tilde{H}_{k-1+r}(\M[2k+1+3r];\Z)$ for 
      $k \le 5$ and $r \le 7$. Legend: $\infty$ = 
      infinite exponent; $p^*$ = exponent an unknown positive power of
      $p$; $e?$ = possibly other prime power exponents than those
      listed.}
   \label{matchingkr2-fig}
  \end{figure}

  We conclude with Table~\ref{matchingkr2-fig}, which provides a list
  of possible exponents in $\tilde{H}_{k-1+r}(\M[2k+1+3r];\Z)$
  for small $k$ and $r$; apply Theorems~\ref{bouc-thm},
  \ref{ShWa-thm}, \ref{m14-thm}, \ref{torsionallover-thm},
  and \ref{p5-thm} and Proposition~\ref{m13-prop}. Note that 
  $(k,r) = (0,2)$ yields the first occurrence of $3$-torsion 
  and that $(k,r) = (2,3)$ yields the only known occurrence of
  $5$-torsion. These two pairs share the property that $k$ is
  maximal for the given $r$ such that the group at $(k,r)$ is finite. 
  Speculating wildly, one may ask whether there is further torsion to
  discover at other pairs $(k,r)$ with this property, that is, $k = 
  \binom{r}{2}-1$; use Theorem~\ref{bouc-thm}.

  \section*{Acknowledgments}

  I thank two anonymous referees for several useful comments.
  This research was carried out at the Technische Universit\"at Berlin
  and at the Massachusetts Institute of Technology in Cambridge, MA.

  \end{document}